\documentclass[11pt]{amsart}

\usepackage[english]{babel}
\usepackage[utf8]{inputenc}
\usepackage[T1]{fontenc}

\usepackage{bm}
\usepackage{microtype}
\usepackage{mathrsfs}
\usepackage[nott,largesmallcaps]{kpfonts}

\usepackage{amsmath}
\usepackage{amssymb}
\usepackage{amsfonts}
\usepackage{amsthm}
\usepackage{mathtools}

\usepackage{mathdots}

\usepackage{xcolor}
\usepackage{graphicx}
\DeclareGraphicsExtensions{.pdf,.png,.jpg}
\usepackage{tikz}
\usetikzlibrary{arrows,calc,decorations.markings,cd}

\usepackage{enumitem}

\usepackage[style=alphabetic,isbn=false,doi=false,url=false,maxalphanames=5,maxbibnames=99]{biblatex}
\DeclareLabelalphaTemplate{
  \labelelement{
    \field[final]{shorthand}
    \field{label}
    \field[strwidth=3,strside=left,ifnames=1]{labelname}
    \field[strwidth=1,strside=left]{labelname}
  }
}
\DeclareLabelalphaNameTemplate{
  \namepart[use=true, pre=true, strwidth=1, compound=true]{prefix}
  \namepart{family}
}
\DeclareFieldFormat{extraalpha}{#1}
\renewbibmacro{in:}{}
\AtEveryBibitem{%
  \clearfield{issn}%
  \ifentrytype{book}{
    \clearfield{pages}%
  }
}
\usepackage{hyperref}

\theoremstyle{definition}
\newtheorem{theoremx}[equation]{Theorem}
\newenvironment{theorem}
{\pushQED{\qed}\theoremx}
{\popQED\endtheoremx}
\newtheorem{definition}[equation]{Definition}

\newtheorem{lemmax}[equation]{Lemma}
\newenvironment{lemma}
{\pushQED{\qed}\lemmax}
{\popQED\endlemmax}
\newtheorem{examplex}[equation]{Example}
\newenvironment{example}
{\pushQED{\qed}\examplex}
{\popQED\endexamplex}
\newtheorem{remarkx}[equation]{Remark}
\newenvironment{remark}
{\pushQED{\qed}\remarkx}
{\popQED\endremarkx}

\DeclareMathOperator{\sgn}{sgn}
\DeclareMathOperator{\SL}{SL}
\DeclareMathOperator{\PSL}{PSL}
\DeclareMathOperator{\AGL}{AGL}
\DeclareMathOperator{\Sym}{Sym}
\DeclareMathOperator{\Alt}{Alt}
\newcommand{\F}{\mathbb{F}}
\newcommand{\jbar}{{\bar\jmath}}

\title{Short presentations of finite simple groups}
\author{Peter Huxford}
\address{Department of Mathematics, University of Auckland,
Auckland, New Zealand}
\email{phux240@aucklanduni.ac.nz}
\begin{document}

\maketitle

\begin{abstract}
  In \cite{GKKL1} Guralnick, Kantor, Kassabov and Lubotzky give 3-generator
  7-relator presentations of $A_n$ and $S_n$ with bit-length $O(\log n)$ for
  $n\geq5$. This is the best possible bit-length, since $\Omega(\log n)$ bits
  are required to specify the integer $n$ in the input. However, the generators
  do not satisfy the relations.

  This paper considers the relevant arguments given in \cite{GKKL1}, identifies
where the errors occur, and shows how they can be fixed in order to recover this
result. The presentations are available in \textsc{Magma}.
\end{abstract}

\tableofcontents

\section{Introduction}

In \cite[p.\ 290]{Sims} Sims observed that ``There is no universal agreement as to when one presentation is simpler than another''. This paper is largely concerned with optimizing one particular measurement. Namely, the \emph{bit-length} of a finite presentation is the number of symbols required to write it down, where each generator is a single symbol, each relator is a string of symbols, and exponents are written in binary.

Observe that bit-length is bounded above by \emph{word length}, where the relators of a finite presentation on generators $X$ are expressed as strings in the alphabet $X\cup X^{-1}$.

The notion of bit-length was introduced by Babai and Szemerédi in \cite{BS}; there they conjectured that, for some constant $c$, every finite simple group $G$ has a presentation of bit-length $O((\log |G|)^c)$. The results of \cite{BGKLP,HS,Suz} show that $c=2$ suffices for all finite simple groups, with the possible exception of the Ree groups ${}^2G_2(q)$.

Presentations with short bit-length have applications in computational group theory. For example, a major project is to compute a composition series of a matrix group defined over a finite field. In \cite{LG} Leedham-Green proposed a randomized algorithm to solve this problem. Given such a group $G$, a binary tree rooted at $G$ is constructed as follows. If $G$ is not simple, then we construct an epimorphism $\theta\colon G\to I$ with nontrivial, proper kernel $K$. Trees rooted at $K$ and $I$ are then recursively constructed and combined to produce a tree rooted at $G$.

The leaves of the resulting tree should be the composition factors of $G$. However, due to the randomized nature of the algorithm, there is a nonzero probability that the proposed composition series may instead be that of a proper subgroup of $G$. To decide whether this is the case, a proposed presentation of $G$ is recursively constructed from presentations of $K$ and $I$ (cf. \cite[Section 10.2]{Joh}). If the relations are satisfied in $G$, then the constructed composition series is correct. To improve the efficiency of this verification step, it is desirable that the presentations of the composition factors, which are finite simple groups, have short bit-length. For more detail on the main algorithm see \cite{BHLO}, and for related surveys see \cite{OB1,OB2}.

Explicit presentations of the alternating and symmetric groups were already known in the 19th century. In \cite{Moo} Moore gave an $(n-1)$-generator $n(n-1)/2$-relator presentation of $S_n$ with generators $(i,i+1)$ for $1\leq i<n$. Namely:
\begin{align}
  \label{eq:moointro}
  \begin{split}
    S_n \cong \langle x_1,\ldots,x_{n-1} &\mid x_i^2=1 \text{ for } 1\leq i<n, \\
    &\phantom{\mid{}} (x_{i-1}x_i)^3 = 1 \text{ for } 1<i<n, \\
    &\phantom{\mid{}} (x_ix_j)^2 = 1 \text{ for } 1\leq i < j-1 < n-1 \rangle.
  \end{split}
\end{align}
This is an example of a \emph{Coxeter} presentation, and has bit-length $O(n^2)$. Moore used this to derive a 2-generator $(n+1)$-relator presentation of $S_n$, which has shorter bit-length $O(n\log n)$, based on a transposition and an $n$-cycle. Introducing auxiliary generators corresponding to powers of the $n$-cycle readily produces a presentation of $S_n$ with bit-length $O(n)$.

In \cite{Car1} (see also \cite[p.\ 169]{Car2}) Carmichael gave an $n$-generator $n(n+1)/2$-relator Coxeter presentation of the simple group $A_{n+2}$ for $n\geq3$ with generators $(i,n+1,n+2)$ for $1\leq i\leq n$. Namely:
\begin{align}
  \label{eq:carmichael}
  \begin{split}
    A_{n+2} \cong \langle x_1,\ldots,x_n &\mid x_i^3 = 1 \text{ for } 1\leq i\leq n, \\
    &\phantom{\mid{}} (x_ix_j)^2 = 1 \text{ for } 1\leq i<j\leq n \rangle.
  \end{split}
\end{align}
Observe this presentation has bit-length $O(n^2)$.

Remarkable progress has recently been made in reducing the asymptotic bit-length of presentations for the finite simple and other related groups. In \cite{BCLO} Bray, Conder, Leedham-Green and O'Brien give presentations of $A_n$ and $S_n$ on a uniformly bounded number of generators and relations with bit-length $O(\log n)$. This is the best possible bit-length, since $\Omega(\log n)$ bits are required to specify $n$ in the input. In \cite{LO} Leedham-Green and O'Brien give presentations of all classical groups of rank $n$ over a field of size $q$ with at most 8 generators, $O(n)$ relations, and bit-length $O(n+\log q)$.

Guralnick, Kantor, Kassabov and Lubotzky prove stronger results in their papers \cite{GKKL2,GKKL1} on presentations of finite simple groups. The goal of the former paper is to optimize asymptotic word length, whereas the latter focuses on bit-length. In \cite{GKKL2} Guralnick \emph{et al.} give presentations of $A_n$ and $S_n$ on a uniformly bounded number of generators and relations with word length $O(\log n)$. In \cite{GKKL1} they show that if $n\geq5$, then $A_n$ and $S_n$ have 3-generator 7-relator presentations with bit-length $O(\log n)$. The correctness of these presentations is crucial for the presentations given in \cite{GKKL1} for other finite simple groups. This includes their major result that all finite quasisimple groups of Lie type of rank $n$ over a field of size $q$, with the possible exception of the Ree groups ${}^2G_2(q)$, have presentations with at most 9 generators, 49 relations, and bit-length $O(\log n+\log q)$. Again, this is the best possible bit-length.

However, the generators of the explicit 3-generator 7-relator presentation of the symmetric group in \cite[Section 3.5]{GKKL1} do not satisfy the relations. The corresponding presentation for the alternating group is not explicitly described. In fact, due to several errors in \cite{GKKL1}, some alternating groups are not covered by the arguments.

In this paper we carefully examine the arguments of \cite{GKKL1} on alternating
and symmetric groups, and identify and correct the relevant errors. We
explicitly describe 3-generator 7-relator presentations of $A_n$ and $S_n$ with
bit-length $O(\log n)$ for all $n$ satisfying either $13\leq n\leq 20$, $25\leq
n\leq 44$, or $n\geq49$. (The remaining values of $n\geq5$ are handled in
\cite{GKKL1}.) \textsc{Magma} \cite{Magma} was used to provide additional
evidence in support of our results.

Much of the work reported here first appeared in \cite{dissertation}, supervised
by Eamonn O'Brien. This report will be updated to reflect other corrections, if any, to the
presentations given in \cite{GKKL1}.

\subsection{Notation}

Functions act on the left. We write $g^h=h^{-1}gh$ and $[g,h]=g^{-1}h^{-1}gh=g^{-1}g^h$. All group actions (including permutations) are right actions. For example $(1,2)(2,3)=(1,3,2)$, and if $(x_1,\ldots,x_k)\in S_n$ is a $k$-cycle, and $\sigma\in S_n$ is an arbitrary permutation, then
\[ (x_1,\ldots,x_k)^\sigma = \sigma^{-1}(x_1,\ldots,x_k)\sigma = (x_1^\sigma,\ldots,x_k^\sigma). \]
If $G$ has a (right) group action on $\Omega$ (that is, $\Omega$ is a $G$-set) and $a\in\Omega$, then the image of $a$ under $g$ is written as $a^g$. This notation agrees with that of \cite{GKKL1}, but permutations are multiplied in the \emph{opposite order} (for example, see \cite[Remark 3.11]{GKKL1}).

\begin{definition}
  Let $G=\langle X \rangle$ be a group and $g\in G$. The \emph{bit-length} of $g$ in $X$ is the number of symbols required to express $g$ in terms of elements of $X$, where each element of $X$ is a single symbol, and exponents are written in binary.
\end{definition}

Thus the bit-length of a finite presentation on generators $X$ is $|X|$ plus the sum of the bit-lengths of the relators in $X$.

\subsection{Preliminaries}

This following is proved in \cite[Lemma 2.3]{GKKL2}.

\begin{lemma}
  \label{lem:2.1}
  Let $\pi\colon F_{X\cup Y}\twoheadrightarrow J=\langle X,Y \mid R,S \rangle$ be the natural surjection, and suppose $H=\langle X\mid R \rangle$ is finite. If $\alpha\colon J\to J_0$ is a homomorphism such that $\alpha\langle \pi(X) \rangle\cong H$, then $\langle \pi(X) \rangle\cong H$.
\end{lemma}

\begin{proof}
  Let $\lambda\colon F_X\twoheadrightarrow H$ be the natural surjection. If $r\in R$, then $r(\pi(X))=1$, so $\lambda$ induces a surjection $H\twoheadrightarrow \langle \pi(X) \rangle$. There is also a surjection $\langle \pi(X) \rangle\twoheadrightarrow\alpha\langle \pi(X) \rangle \cong H$.
\end{proof}

\begin{definition}
  \label{def:2-homogeneous}
  Let a group $G$ act on a set $\Omega$. The action is \emph{2-homogeneous} if the induced action of $G$ on the 2-subsets of $\Omega$ is transitive.
\end{definition}

\begin{lemma}
  \label{lem:2-homogeneous-implies-transitive}
  If $|\Omega|\geq3$, then a 2-homogeneous action on $\Omega$ is transitive.
\end{lemma}

\begin{proof}
  Let $G$ act 2-homogeneously on $\Omega$. Let $a,b,c$ be distinct elements of $\Omega$, and let $g,h\in G$ satisfy  $\{a,c\}^g=\{b,c\}$ and $\{a,c\}^h=\{a,b\}$. If $a^g,a^h\neq b$, then $a^g=c$ and $c^h=b$, hence $a^{gh}=b$.
\end{proof}

\begin{lemma}
  \label{lem:2-homogeneous-alternating}
  Let $G$ act 2-homogeneously on $\Omega$, and let $\phi\colon G\to\Sym(\Omega)$ be the corresponding homomorphism. Suppose $|\Omega|\geq3$.
  \begin{enumerate}[label=(\roman*)]
  \item If the action is not 2-transitive, then $\phi(G)\leq\Alt(\Omega)$.
  \item The action of $\phi(G)\cap\Alt(\Omega)$ on $\Omega$ is transitive.\qedhere
  \end{enumerate}
\end{lemma}

\begin{proof}
  (i) Note that $G$ contains no involutions, otherwise some ordered pair of distinct elements of $\Omega$ could be sent to every other such pair. Hence $G$ has no elements of even order, thus $\phi(G)\leq\Alt(\Omega)$.

  (ii) If $\phi(G)\leq\Alt(\Omega)$, then this follows from Lemma~\ref{lem:2-homogeneous-implies-transitive}. Otherwise, by (i) the action is 2-transitive. Note $[\phi(G):\phi(G)\cap\Alt(\Omega)]\leq[\Sym(\Omega):\Alt(\Omega)]=2$, and $|\phi(G)|\geq|\Omega|$ since $G$ acts transitively on $\Omega$ by Lemma~\ref{lem:2-homogeneous-implies-transitive}. Therefore $|\phi(G)\cap\Alt(\Omega)|\geq\frac{1}{2}|\phi(G)|\geq\frac{1}{2}|\Omega|>1$.

  Now let $h\in G$ with $\phi(h)\in\Alt(\Omega)\setminus\{1\}$. Let $a,c\in\Omega$ be such that $a^h=c\neq a$. For each $b\in\Omega\setminus\{a\}$, by 2-transitivity there is a $g\in G$ such that $a^g=a$ and $c^g=b$. Then $a^{g^{-1}hg}=b$, and $\phi(g^{-1}hg)\in\Alt(\Omega)$. Hence $\phi(G)\cap\Alt(\Omega)$ is transitive.
\end{proof}

\begin{definition}
  An \emph{isomorphism} $f\colon\Omega\to\Gamma$ of $G$-sets is a bijection such that $f(a^g)=f(a)^g$ for all $a\in\Omega$ and $g\in G$. Note that $f^{-1}$ is also an isomorphism of $G$-sets. If such an isomorphism exists, then $G$ \emph{acts on $\Gamma$ as it does on $\Omega$}.
\end{definition}

\begin{lemma}
  \label{lem:orbit-stabilizer}
  If $G$ acts transitively on $\Omega$ and $a\in\Omega$, then $G$ acts on $\Omega$ as it does on right cosets $G/G_a$, where $G_a$ is the stabilizer of $a$. In particular there is an isomorphism of $G$-sets $f\colon\Omega\to G/G_a$ given by $f(a^g)\coloneqq G_ag$ for $g\in G$.
\end{lemma}

\begin{proof}
  If $g,h\in G$, then $a^g=a^h$ is equivalent to $gh^{-1}\in G_a$, and hence also $G_ag=G_ah$. Thus $f$ is a well-defined bijection. Moreover, $f(a^{gh})=G_agh=f(a^g)h$ for all $g,h\in G$. Therefore $f$ is an isomorphism of $G$-sets, so $G$ acts on $\Omega$ as it does on $G/G_a$.
\end{proof}

\begin{lemma}
  \label{lem:same-action}
  Suppose $G$ acts transitively on finite sets $\Omega$, $\Gamma$ with $|\Omega|=|\Gamma|$, and $G_a\subseteq G_b$ for some $a\in\Omega$, $b\in\Gamma$. Then $G$ acts on $\Gamma$ as it does on $\Omega$. More precisely, there is exactly one isomorphism $f\colon\Omega\to\Gamma$ of $G$-sets which satisfies $f(a)=b$.
\end{lemma}

\begin{proof}
  The transitivity of the $G$-set $\Omega$ uniquely determines such an isomorphism. Lemma~\ref{lem:orbit-stabilizer} implies that $\Omega\cong G/G_a$ and $\Gamma\cong G/G_b$ as $G$-sets. Thus $[G:G_a]=[G:G_b]$, so $G_a=G_b$, and hence $\Omega\cong\Gamma$ as $G$-sets. The isomorphism obtained clearly maps $a$ to $b$.
\end{proof}

\section{Alternating and symmetric groups}

We follow the strategy employed by Guralnick \emph{et al.} \cite{GKKL1} to construct presentations of alternating and symmetric groups for sufficiently large degrees.

For appropriately chosen 2-homogeneous groups $T$ and certain primes $p$, we construct 3-generator 5-relator presentations of $A_{p+2}\times T$ and $S_{p+2}\times T$. Introducing an additional relator gives 3-generator 6-relator presentations of $A_{p+2}$ and $S_{p+2}$. If $p\equiv11\bmod{12}$, then we reduce the number of generators by one and omit two relations, thereby producing 2-generator 4-relator presentations of $A_{p+2}$ and $S_{p+2}$.

We `glue' these presentations to obtain 3-generator 8-relator presentations of other alternating and symmetric groups; a strengthening of Bertrand's postulate shows that this handles all degrees at least 51. If care is taken when `gluing', then we can combine two of the relations, which yields 3-generator 7-relator presentations of these alternating and symmetric groups.

\subsection{Using 2-homogeneous groups for special degrees}

The following is \cite[Lemma 3.2]{GKKL1}.

\begin{lemma}
  \label{2-homogeneous}
  Let $n\geq3$ be an integer, let $T$ be a group with a \emph{2-homogeneous} action on $\{1,\ldots,n\}$, and let $\bar{\phantom{T}}\colon T\twoheadrightarrow\bar{T}\leq S_n$ be the induced surjection. Suppose $T=\langle X \mid R \rangle$, and $T_1=\langle X_1 \rangle$ is the stabilizer of 1, viewing the elements of $X_1$ as words in $X$. Let $w$ be a word in $X$ that, viewed as an element of $T$, moves 1 and satisfies $\bar{w}\in A_n$. View $S_n$ as a subgroup of $S_{n+2}$ fixing the points $n+1$ and $n+2$. For every $\sigma\in S_n$ write $\sgn(\sigma)=(-1)^{\epsilon(\sigma)}$, where $\epsilon(\sigma)\in\{0,1\}$. Then
  \[ J \coloneqq \langle X, z \mid R,\ z^3 = (zz^w)^2 = 1,\ z^t=z^{\sgn(\bar{t})} \text{ for } t\in X_1 \rangle \cong A_{n+2} \times T, \]
  via the mapping $\varphi\colon X\cup\{z\}\to A_{n+2}\times T$ defined by
  \begin{align*}
      \varphi(x) &\coloneqq (\bar{x}\cdot(n+1,n+2)^{\epsilon(\bar{x})}, x) \quad \text{for } x\in X, \\
      \varphi(z) &\coloneqq ((1,n+1,n+2), 1).\qedhere
  \end{align*}
\end{lemma}

\begin{proof}
  The image of $\varphi$ satisfies the defining relations of $J$, hence $\varphi$ extends to a homomorphism $\varphi\colon J\to A_{n+2}\times T$.

  We first prove that $\varphi$ is a surjection. Projecting onto the second coordinate induces an isomorphism $\langle \varphi(X) \rangle\cong T$. Moreover, $\langle \varphi(X) \rangle$ acts 2-homogeneously as $\bar{T}$ on $\{1,\ldots,n\}$ when restricted to the first component. By Lemma~\ref{lem:2-homogeneous-implies-transitive}, for each $1\leq i\leq n$ there exist $(\sigma,t)\in\langle \varphi(X) \rangle$ with $1^{\sigma}=i$, so $\varphi(z)^{(\sigma,t)}=((i,n+1,n+2),1)$. Hence $\langle \varphi(z)^{\langle \varphi(X) \rangle} \rangle=A_{n+2}\times1$. Thus $\varphi(J)$ contains $(1,x)$ for each $x\in X$, since it contains $(\bar{x}(n+1,n+2)^{\epsilon(\bar{x})},1)$ for each $x\in X$. Therefore $\varphi$ is surjective. This also defines a surjection $\pi\colon J\to A_{n+2}$, by projecting onto the first coordinate.

  Since $\langle \varphi(X) \rangle\cong T$, by Lemma~\ref{lem:2.1} we may identify the subgroup $\langle X \rangle$ of $J$ with $T$. We seek to prove that $N\coloneqq\langle z^T \rangle\cong A_{n+2}$ by verifying the relations in \eqref{eq:carmichael}, using the defining relations $z^3=(zz^w)^2=1$ of $J$. Note that both $T$ and $\bar{T}\cap A_n$ act transitively on $\{1,\ldots,n\}$. (This follows from Lemmas~\ref{lem:2-homogeneous-implies-transitive} and \ref{lem:2-homogeneous-alternating}(ii).)

  The relations $z^t=z^{\sgn(\bar{t})}$ for $t\in X_1$ also hold for $t\in T_1$, hence $\langle z \rangle^{T_1}=\langle z \rangle$. By the orbit-stabilizer theorem, $[T:T_1]=n$, hence $|\langle z \rangle^T|\leq n$. Since $\pi(\langle z \rangle^T)=\{\langle (i,n+1,n+1) \rangle : 1\leq i\leq n\}$ has $n$ elements, $|\langle z \rangle^T|=n$. By Lemma~\ref{lem:same-action}, $T$ acts on $\langle z \rangle^T$ as it does on $\{1,\ldots,n\}$, and $T_1$ is the stabilizer of $\langle z \rangle$.

  If $t\in T$ and $\bar{t}=1$, then $\bar{g}\bar{t}\bar{g}^{-1}=1$ for $g\in T$. Since $\sgn(\bar{g}\bar{t}\bar{g}^{-1})=1$ and $gtg^{-1}\in T_1$, our relations imply that $z^{gtg^{-1}}=z$, so $(z^g)^t=z^g$. Thus $T$ induces a well-defined action of $\bar{T}$ on $z^T$, and hence also on $\langle z \rangle^T$.

  Consider the actions of $\bar{T}\cap A_n$ on $z^T$ and $\langle z \rangle^T$. The stabilizer of $\langle z \rangle$ is $\bar{T}_1\cap A_n$. Since $z^t=z^{\sgn(t)}$ for $t\in T_1$, each element of $\bar{T}_1\cap A_n$ fixes $z$. Thus $z^{\bar{T}\cap A_n}\cap\langle z \rangle=\{z\}$. Recall that $\bar{T}\cap A_n$ acts transitively on $\langle z \rangle^T$, hence $z^{\bar{T}\cap A_n}$ has exactly one element in common with each element of $\langle z \rangle^T$. Therefore $|z^{\bar{T}\cap A_n}|=n$. By Lemma~\ref{lem:same-action}, $\bar{T}\cap A_n$ acts on $z^{\bar{T}\cap A_n}$ as it does on $\langle z \rangle^{\bar{T}\cap A_n}=\langle z \rangle^T$ and $\{1,\ldots,n\}$, with $\bar{T}_1\cap A_n$ being the stabilizer of $z$.

  If $\bar{T}\not\leq A_n$, then since $\bar{T}\cap A_n$ is transitive, $\bar{T}_1\not\leq A_n$, and thus $z^t=z^{-1}$ for some $t\in T_1$. Therefore $|z^T|=2n$, and $\bar{T}\cap A_n$ has two orbits on $z^T$, namely $z^{\bar{T}\cap A_n}$ and $(z^{-1})^{\bar{T}\cap A_n}$. Otherwise, if $\bar{T}\leq A_n$, then $z^{\bar{T}\cap A_n}=z^T$. In each case $N=\langle z^{\bar{T}\cap A_n} \rangle=\langle z^T \rangle$.

  If $\bar{T}\cap A_n$ acts 2-homogeneously on $z^{\bar{T}\cap A_n}$, then each 2-subset of $z^{\bar{T}\cap A_n}$ is $(\bar{T}\cap A_n)$-conjugate to $\{z,z^w\}$. The relation $(zz^w)^2=1$ also implies $(z^wz)^2=1$, hence $z^{\bar{T}\cap A_n}$ satisfies the relations in \eqref{eq:carmichael}. Thus $N=\langle z^{\bar{T}\cap A_n} \rangle\cong A_{n+2}$ by the simplicity of $A_{n+2}$.

  If $\bar{T}\cap A_n$ does not act 2-homogeneously on $z^{\bar{T}\cap A_n}$ (equivalently $\langle z \rangle^T$), then $\bar{T}\not\leq A_n$, so by Lemma~\ref{lem:2-homogeneous-alternating}(i) $\bar{T}$ acts 2-transitively on $\langle z \rangle^T$. We claim that $N\cong A_{n+2}$.

  Note that $|\bar{T}:\bar{T}\cap A_n|\leq|S_n:A_n|=2$. Since $\bar{T}$ acts transitively on the ordered pairs of distinct elements of $\langle z \rangle^T$, it follows that $\bar{T}\cap A_n$ has at most 2 orbits under this action. In fact there are exactly 2 orbits, since $\bar{T}\cap A_n$ is not 2-homogeneous. Recall that $\bar{T}\cap A_n$ acts transitively on $\langle z \rangle^T$, hence we may choose representatives $(\langle z \rangle, \langle z \rangle^w)$ and $(\langle z \rangle, \langle z \rangle^y)$ of the two orbits for some $y\in\bar{T}\cap A_n$.

  Let $g\in\bar{T}\setminus A_n$ satisfy $(\langle z \rangle, \langle z \rangle^w)^g=(\langle z \rangle, \langle z \rangle^y)$. Since $g\notin A_n$ and $z,z^w\in z^{\bar{T}\cap A_n}$, both $z^g$ and $(z^w)^g$ lie in the other $(\bar{T}\cap A_n)$-class $(z^{-1})^{\bar{T}\cap A_n}$. Thus $z^g=z^{-1}$, and $(z^w)^g=(z^y)^{-1}$ since $z^y\in z^{\bar{T}\cap A_n}$, so $(zz^y)^2=((z^wz)^{-2})^g=1$. Hence $z^{\bar{T}\cap A_n}$ satisfies the relations in \eqref{eq:carmichael}, thus $N\cong A_{n+2}$.

  In all cases $N\cong A_{n+2}$. Clearly $N\unlhd J$ and $J/N=J/\langle z^T \rangle\cong T$. Thus $|J|=|A_{n+2}\times T|$, so $\varphi$ is an isomorphism.
\end{proof}

We elaborate on some applications of Lemma~\ref{2-homogeneous} outlined in \cite[Examples 3.4]{GKKL1}. To produce a presentation of $A_{n+2}$, we find an appropriate group $T$ and apply Lemma~\ref{2-homogeneous}. This produces a presentation which defines a group $\hat{J}\cong A_{n+2}\times T$. By introducing an additional relator whose normal closure in $\hat{J}$ corresponds to the subgroup $1\times T$, we obtain a new presentation which defines a group $J\cong A_{n+2}$.

Our first example is a discussion of presentations of $A_{p+3}\times\SL(2,p)$ and $A_{p+3}$ for primes $p>3$, as described in \cite[Examples 3.4 (1), Examples 3.18 (1)]{GKKL1}.

\begin{example}
  \label{Examples3.18(1)}  
  Let $p>3$ be prime. Now $T\coloneqq\SL(2,p)$ acts 2-homogeneously (in fact, 2-transitively) on the $n\coloneqq p+1$ \emph{projective points} in $\F_p\times\F_p$. Identify the subspace spanned by $(1,\alpha)$ with $\alpha$ for each $\alpha\in\F_p$, and the subspace spanned by $(0,1)$ with the point $\infty$, so that $T$ acts on $\F_p\cup\{\infty\}$. Let $\bar{\phantom{T}}\colon T\to\Sym(\F_p\cup\{\infty\})$ be the induced map. The point $\infty$ plays the role of the point 1 in Lemma~$\ref{2-homogeneous}$. Using \cite{Sun}, Campbell and Robertson \cite{CR2} obtain a 2-generator 2-relator presentation  of $T$. Namely:
  \begin{equation}
    \label{eq:SL(2,p)}
    T = \SL(2,p) = \langle x, y \mid x^2 = (xy)^3,\ (xy^4xy^{(p+1)/2})^2y^px^{2\lfloor p/3 \rfloor} = 1 \rangle.
  \end{equation}
  Matrices in $\SL(2,p)$ that satisfy the relations are
  \begin{equation}
    t \coloneqq (-1)^k\begin{pmatrix} 0 & -1 \\ 1 & 0 \end{pmatrix}, \quad u \coloneqq (-1)^k\begin{pmatrix} 1 & 1 \\ 0 & 1 \end{pmatrix},
    \label{eq:SL(2,p)-generators}
  \end{equation}
  where $p\equiv k\bmod{3}$ and $k\in\{1,2\}$. Moreover, $\SL(2,p)=\langle t,u \rangle$.

  The involution $\bar{t}$ maps $\alpha$ to $-\alpha^{-1}$ for $\alpha\in\F_p^*$ and swaps $0$ with $\infty$. The fixed points of $\bar{t}$ are precisely the square roots of $-1$ in $\F_p^*$. Thus if $p\equiv1\bmod{4}$, then $\bar{t}$ is a product of $(p-1)/2$ transpositions, and if $p\equiv3\bmod{4}$, then $\bar{t}$ is a product of $(p+1)/2$ transpositions. In either case, $\sgn(\bar{t})=1$ and $\epsilon(\bar{t})=0$. The $p$-cycle $\bar{u}=(0,1,\ldots,p-1)$ satisfies $\sgn(\bar{u})=1$ and $\epsilon(\bar{u})=0$.

  Let $\F_p^*=\langle j \rangle$ and $j\jbar\equiv1\bmod{p}$. Let $k$ be as in \eqref{eq:SL(2,p)-generators}. Define the word
  \begin{equation}
    \label{eq:p+3-h}
    h\coloneqq y^\jbar(y^j)^xy^\jbar x^{(-1)^k}.
  \end{equation}
  Let $v\coloneqq u^\jbar(u^j)^tu^\jbar t^{(-1)^k}=(-1)^{jk}\left( \begin{smallmatrix} \jbar & 0 \\ 0 & j \end{smallmatrix} \right)$ be the corresponding element of $\SL(2,p)$. Replacing $j$ by $j-p$ if necessary, we may assume $jk$ is even, so $v=\left( \begin{smallmatrix} \jbar & 0 \\ 0 & j \end{smallmatrix} \right)$. Thus $\bar{v}$ maps $\alpha$ to $j^2\alpha$ for $\alpha\in\F_p^*$ and fixes 0 and $\infty$, so $\bar{v}$ is the product of two disjoint $(p-1)/2$-cycles.

  Now $\langle v \rangle=\left\{\left( \begin{smallmatrix} a^{-1} & 0 \\ 0 & a \end{smallmatrix} \right) : a\in\F_p^*\right\}$, hence $\left( \begin{smallmatrix} 1 & 1 \\ 0 & 1 \end{smallmatrix} \right)=(-1)^ku\in\langle u,v \rangle$. Since $\left( \begin{smallmatrix} a^{-1} & 0 \\ 0 & a \end{smallmatrix}\right)\left( \begin{smallmatrix} 1 & 1 \\ 0 & 1 \end{smallmatrix} \right)^{ab}=\left( \begin{smallmatrix} a^{-1} & b \\ 0 & a \end{smallmatrix} \right)$ and $T_\infty=\left\{ \left( \begin{smallmatrix} a^{-1} & b \\ 0 & a \end{smallmatrix} \right) : a\in\F_p^*,\ b\in\F_p \right\}$, it follows that $T_\infty=\langle u,v \rangle$. (Recall that $\infty$ plays the role of 1 in Lemma~\ref{2-homogeneous}.) Thus we may take $X_\infty=\{y,h\}$. By Lemma~\ref{2-homogeneous} there is a 3-generator 6-relator presentation of $A_{p+3}\times T$. More precisely, let
  \begin{align}
    \label{AltSL2}
    \hat{J} \coloneqq \langle x,y,z \mid &\ x^2 = (xy)^3,\ (xy^4xy^{(p+1)/2})^2y^px^{2\lfloor p/3 \rfloor} = z^3 = (zz^x)^2 = [y,z] = [h,z] = 1 \rangle,
  \end{align}
  where $h$ is defined as in \eqref{eq:p+3-h}. Let $\Omega\coloneqq\F_p\cup\{\infty,\star,\bullet\}$. Define $\hat{\varphi}\colon\{x,y,z\}\to\Alt(\Omega)\times T$ by
  \begin{align}
    \label{eq:Examples3.18(1)}
    \begin{split}
      \hat{\varphi}(x) &\coloneqq \left( \bar{t}, t \right), \\
      \hat{\varphi}(y) &\coloneqq \left( \bar{u}, u \right), \\
      \hat{\varphi}(z) &\coloneqq \left( (\infty,\star,\bullet), 1 \right),
    \end{split}
  \end{align}
  where $t,u$ are the matrices defined in \eqref{eq:SL(2,p)-generators}, and $\bar{t},\bar{u}$ are viewed as permutations on $\Omega$ that fix $\star$ and $\bullet$. Lemma~\ref{2-homogeneous} asserts that $\hat{\varphi}$ extends to an isomorphism $\hat{\varphi}\colon \hat{J}\to\Alt(\Omega)\times T$.

  Let $J$ be the quotient of $\hat{J}$ obtained from \eqref{AltSL2} by imposing the relation $(hz^{xy}z^{xy^j})^{(p+1)/2}=1$. We claim that $J\cong A_{p+3}$. Notice that $\hat{\varphi}(h)=\left( \bar{v}, v \right)$. Furthermore, $\bar{t}\bar{u}$ maps $\infty\mapsto0\mapsto1$ and $\bar{t}\bar{u}^j$ maps $\infty\mapsto0\mapsto j$, while both fix $\star,\bullet$. Hence the first coordinate of $\hat{\varphi}(z^{xy}z^{xy^j})$ is $(\infty,\star,\bullet)^{\bar{t}\bar{u}}(\infty,\star,\bullet)^{\bar{t}\bar{u}^j}=(1,\star,\bullet)(j,\star,\bullet)=(1,\bullet)(j,\star)$. Note that $1$ and $j$ appear in distinct cycles of $\bar{v}$, hence $\bar{v}(1,\bullet)(j,\star)$ is the product of two disjoint $(p+1)/2$-cycles. Therefore
  \[ \hat{\varphi}((hz^{xy}z^{xy^j})^{(p+1)/2}) = (1,v^{(p+1)/2}) = \left( 1, \begin{pmatrix} \jbar^{(p+1)/2} & 0 \\ 0 & j^{(p+1)/2} \end{pmatrix} \right). \]
  Recall that $\PSL(2,p)=T/\{\pm I\}$ is simple for primes $p>3$. Hence, the normal closure of $v^{(p+1)/2}$ in $T$ is $T$. Thus adding the relator $(hz^{xy}z^{xy^j})^{(p+1)/2}$ to the presentation in \eqref{AltSL2} yields a 3-generator 7-relator presentation of $\Alt(\Omega)$. Upon selecting a bijection between $\Omega$ and $\{1,\ldots,p+3\}$, an isomorphism $\Alt(\Omega)\cong A_{p+3}$ is obtained, thus
  \begin{align}
    \label{eq:alt-p+3}
    \begin{split}
      J \coloneqq \langle x,y,z \mid &\ x^2 = (xy)^3,\ (xy^4xy^{(p+1)/2})^2y^px^{2\lfloor p/3 \rfloor} = 1, \\
      &\ z^3 = (zz^x)^2 = [y,z] = [h,z] = 1,\ (hz^{xy}z^{xy^j})^{(p+1)/2} = 1 \rangle \cong A_{p+3},
    \end{split}
  \end{align}
  where $h$ is defined as in \eqref{eq:p+3-h}. In view of the above argument and \eqref{eq:Examples3.18(1)}, an isomorphism $\varphi\colon J\to\Alt(\Omega)$ is defined by $\varphi(x)\coloneqq\bar{t}$, $\varphi(y)\coloneqq\bar{u}$ and $\varphi(z)\coloneqq(\infty,\star,\bullet)$.
\end{example}

\begin{remark}
  \label{rem:p+3-errors}
  Example~\ref{Examples3.18(1)} is based on \cite[Examples 3.18(1)]{GKKL1}, however the latter contains some errors. The matrices $t,u$ in \eqref{eq:SL(2,p)-generators} differ from those defined in \cite{GKKL1}, which we call $t'\coloneqq\left(\begin{smallmatrix}0 & 1 \\ -1 & 0\end{smallmatrix}\right)$ and $u'\coloneqq\left( \begin{smallmatrix} 1 & 1 \\ 0 & 1 \end{smallmatrix} \right)$ respectively. Specifically, $t=(-1)^{k+1}t'$ and $u=(-1)^ku'$. Observe that $t$ and $u$ are scalar multiples of $t'$ and $u'$ respectively, so $\bar{t}=\bar{t}'$ and $\bar{u}=\bar{u}'$ in $\bar{T}$. However, $t'$ and $u'$ do not satisfy the relations in \eqref{eq:SL(2,p)}. Indeed, matrix calculations demonstrate that
  \[ (t')^2 = -(t'u')^3, \quad (t'u'^4t'u'^{(p+1)/2})^2u'^pt'^{2\lfloor p/3 \rfloor} = (-1)^{\lfloor p/3 \rfloor}I_2. \]
  This calculation, together with the observation that when $p$ is an odd prime, $\lfloor p/3 \rfloor$ is even if and only if $p\equiv1\bmod{3}$, explains the presence of the factors of $(-1)^k$ in \eqref{eq:SL(2,p)-generators}.

  In \cite[Examples 3.18 (1)]{GKKL1} a word $h'\coloneqq y^\jbar(y^j)^xy^\jbar x^{-1}$ is used instead of $h$ (cf. \eqref{eq:p+3-h}), and the relator $(h'z^{yx}z^{y^jx})^{(p+1)/2}$ is used instead of $(hz^{xy}z^{xy^j})^{(p+1)/2}$ (cf. \eqref{eq:alt-p+3}). We require $jk$ to be even in Example~\ref{Examples3.18(1)} (see the calculation of $v$), but no comment is made about the parity of $jk$ in \cite{GKKL1}.

  Let $v'\coloneqq u^\jbar(u^j)^tu^\jbar t^{-1}$ be the element of $\SL(2,p)$ corresponding to $h'$. We compute the permutation $\bar{v}'(\infty,\star,\bullet)^{\bar{u}\bar{t}}(\infty,\star,\bullet)^{\bar{u}^j\bar{t}}$ in $\Alt(\Omega)$ corresponding to $h'z^{yx}z^{y^jx}$. Note that $\bar{v}=\bar{v}'$ in $\Alt(\Omega)$. Since $\bar{u},\bar{v}$ fix $\infty$ and $\bar{t}$ swaps $0$ and $\infty$, it follows that $\bar{v}'(\infty,\star,\bullet)^{\bar{u}\bar{t}}(\infty,\star,\bullet)^{\bar{u}^j\bar{t}}=\bar{v}(0,\star,\bullet)^2$. This is the product of two disjoint cycles of lengths $(p-1)/2$ and $(p+3)/2$, and does not have order dividing $(p+1)/2$.

  Thus the final relation is not satisfied by the proposed generators of $\Alt(\Omega)$. Furthermore, a coset enumeration on the first presentation of $A_{p+3}$ provided in \cite[Examples 3.18 (1)]{GKKL1} yields the trivial group for small values of $p>3$.

  It is not sufficient to correct the value of $h'$. Since $\bar{v}=\bar{v}'$ in $\Alt(\Omega)$, the proposed generators of $\Alt(\Omega)$ also do not satisfy the relation $(hz^{yx}z^{y^jx})^{(p+1)/2}=1$. One must use the corrected relator $(hz^{xy}z^{xy^j})^{(p+1)/2}$ as in \eqref{eq:alt-p+3}.

  It is also not sufficient to correct the order of exponents in the final relator without correcting $h'$. Although $v'$ stabilizes $\infty$, in general $\langle u,v' \rangle$ is not equal to $T_\infty$, and may instead be a proper subgroup. This is because $v'$ may differ from $v$ by a sign. In fact, if $k$ is even, then $v'=-v$, regardless of which integer $j$ is chosen with $\F_p^*=\langle j \rangle$. Furthermore, if $p\equiv3\bmod{4}$, then $-j$ has half the order of $j$ in $\F_p^*$. Hence if $p\equiv11\bmod{12}$, then $\langle u,v' \rangle$ is a subgroup of index two in $T_\infty$.

  A computer search verified that in $\SL(2,11)$ and $\SL(2,23)$ there are no matrices $x,y$ satisfying the relations in $\eqref{eq:SL(2,p)}$ such that $\langle y, h'\rangle$ has the same order as $T_\infty$ (here we used $j=2,5$ and $\jbar=6,14$ for $p=11,23$ respectively). Let $p=11$, $j=2$, $\jbar=6$, and let $J'$ be the group defined by \eqref{eq:alt-p+3} with the word $h$ replaced by $h'$. Let
  \[ H' \coloneqq \langle 
  zxz^{-1},\
  (z^{-1}x^{-1}z)^4,\
  z^{-1}x^2zxz^{-1}x^{-2}z,\
  z^{-1}xy^2xzx^{-1}y^{-2}x^{-1}z \rangle \leq J'.
  \]
  Coset enumeration shows that $[J':H']=28$. The induced permutation representation $J'\to S_{28}$ on cosets has image $2^{13}.\ A_{14}$. In particular $J'$ is not isomorphic to $A_{14}$.
\end{remark}

The following example discusses presentations of $A_{p+2}\times\AGL(1,p)$ and $A_{p+2}$ for odd primes $p$, as described in \cite[Examples 3.4 (2), Corollary 3.7]{GKKL1}.

\begin{example}
  \label{ex:Examples 3.4 (2)}
  Let $p$ be an odd prime, and let
  \[ T \coloneqq \AGL(1,p)=\{x\mapsto\alpha x+\beta \mid \alpha\in\F_p^*,\beta\in\F_p\}. \]
  Then $T$ acts $2$-homogeneously (in fact, 2-transitively) on $\F_p$. Let $\F_p^*=\langle r \rangle$ and let $\bar{a},\bar{b}\in T$ be the cycles of order $p$ and $p-1$ respectively, defined by
  \begin{equation}
    \label{eq:AGL-gens}
    \bar{a} \coloneqq (0,1,\ldots,p-1), \qquad \bar{b}\colon x\mapsto rx.
  \end{equation}
  Then $T=\langle \bar{a} \rangle\rtimes\langle \bar{b} \rangle$. Furthermore, the stabilizer $T_0$ of 0 is $\langle \bar{b} \rangle$. Here $0$ plays the role of the point 1 in Lemma~\ref{2-homogeneous}.
  Note that \cite{GKKL1} uses $x\mapsto r^{-1}x$ instead of $x\mapsto rx$; this will not satisfy the relations discussed here since permutations are multiplied in the opposite order.

  By \cite{Neu}, there is a 2-generator 2-relator presentation of $T=\AGL(1,p)$. Specifically, if $s(r-1)\equiv-1\bmod{p}$, then
  \begin{equation}
    \label{eq:AGL(1,p)}
    T = \AGL(1,p) = \langle a,b \mid a^p = b^{p-1},\ (a^s)^b = a^{s-1} \rangle,
  \end{equation}
  where $a,b$ correspond to the generators $\bar{a},\bar{b}$ respectively.
  Now $\sgn(\bar{a})=1$, $\epsilon(\bar{a})=0$, $\sgn(\bar{b})=-1$ and $\epsilon(\bar{b})=1$. By Lemma~\ref{2-homogeneous}, there is a 3-generator 5-relator presentation of $A_{p+2}\times T$. More precisely, let
  \begin{equation}
    \label{eq:Jhat-AGL-p+2}
    \hat{J} \coloneqq \langle a,b,z \mid a^p = b^{p-1},\ (a^s)^b = a^{s-1},\ z^3 = (zz^a)^2 = 1,\ z^b = z^{-1} \rangle.
  \end{equation}
  Let $\Omega\coloneqq\F_p\cup\{\star,\bullet\}$. Define $\hat{\varphi}\colon\{a,b,z\}\to\Alt(\Omega)\times T$ by
  \begin{align}
    \label{eq:Examples 3.4 (2)}
    \begin{split}
      \hat{\varphi}(a) &\coloneqq (\bar{a},\bar{a}), \\
      \hat{\varphi}(b) &\coloneqq (\bar{b}(\star,\bullet),\bar{b}), \\
      \hat{\varphi}(z) &\coloneqq ((0,\star,\bullet),1),
    \end{split}
  \end{align}
  where in the left coordinate $\bar{a},\bar{b}$ are viewed as permutations on $\Omega$ that fix $\star$ and $\bullet$. Lemma~\ref{2-homogeneous} asserts that $\hat{\varphi}$ extends to an isomorphism $\hat{\varphi}\colon \hat{J}\to\Alt(\Omega)\times T$.

  Let $J$ be the quotient of $\hat{J}$ obtained from \eqref{eq:Jhat-AGL-p+2} by imposing the relation $(b^az)^p=1$. We claim that $J\cong A_{p+2}$. Observe that $\hat{\varphi}(b^az)=(\bar{b}^{\bar{a}}(\star,\bullet),\bar{b}^{\bar{a}})((0,\star,\bullet),1)=(\bar{c},\bar{b}^{\bar{a}})$ for a $p$-cycle $\bar{c}$. Thus $\hat{\varphi}((b^az)^p)=(1,\bar{b}^{\bar{a}})$. The normal closure of $\bar{b}^{\bar{a}}$ in $T$ is $T$. Hence, adding the relator $(b^az)^p$ to the presentation in \eqref{eq:Jhat-AGL-p+2} yields a presentation of $\Alt(\Omega)$. Upon selecting a bijection between $\Omega$ and $\{1,\ldots,p+2\}$, an isomorphism $\Alt(\Omega)\cong A_{p+2}$ is obtained, thus
  \begin{equation}
    \label{eq:Corollary-3.7}
    J \coloneqq \langle a,b,z \mid a^p=b^{p-1},\ (a^s)^b = a^{s-1},\ z^3=(zz^a)^2=1,\ z^b=z^{-1},\ (b^az)^p = 1 \rangle \cong A_{p+2}.
  \end{equation}
  In view of the above argument and \eqref{eq:Examples 3.4 (2)}, an isomorphism $\varphi\colon J\to\Alt(\Omega)$ is defined by $\varphi(a)\coloneqq\bar{a}$, $\varphi(b)\coloneqq\bar{b}(\star,\bullet)$ and $\varphi(z)\coloneqq(0,\star,\bullet)$.
\end{example}

The previous two examples apply Lemma~\ref{2-homogeneous} with a 2-transitive group $T$. Next we let $T$ be an index two subgroup of $\AGL(1,p)$ for primes $p\equiv3\bmod{4}$ where $p>3$ that is 2-homogeneous but not 2-transitive on $\F_p$. The following example discusses presentations of $A_{p+2}\times T$ and $A_{p+2}$, as described in \cite[Examples 3.4 (3), Corollary 3.8]{GKKL1}.

\begin{example}
  \label{ex:Examples 3.4 (3)}
  Let $p\equiv3\bmod{4}$ where $p>3$ is prime, and let
  \[ T\coloneqq\AGL(1,p)^{(2)} = \{ x \mapsto \alpha x + \beta \mid \alpha\in\F_p^{*2}, \beta\in\F_p\}. \]
  Then $T$ acts 2-homogeneously on $\F_p$. Let $\F_p^{*2}=\langle r \rangle$, and let $\bar{a},\bar{b}\in T$ be the permutations of order $p$ and $(p-1)/2$ respectively defined by \eqref{eq:AGL-gens}. Observe that $T=\langle \bar{a} \rangle\rtimes\langle \bar{b} \rangle$. Furthermore, the stabilizer $T_0$ of 0 is $\langle \bar{b} \rangle$. Here 0 plays the role of the point 1 in Lemma~\ref{2-homogeneous}. By \cite{Neu}, $T$ has a 2-generator 2-relator presentation. Specifically, if $s(r-1)\equiv-1\bmod{p}$, then
  \begin{equation}
    \label{eq:AGL(1,p)^(2)}
    T = \AGL(1,p)^{(2)} = \langle a,b \mid a^p = b^{(p-1)/2},\ (a^s)^b = a^{s-1} \rangle,
  \end{equation}
  where $a,b$ correspond to the generators $\bar{a},\bar{b}$ respectively. Now $\sgn(\bar{a})=\sgn(\bar{b})=1$, $\epsilon(\bar{a})=\epsilon(\bar{b})=0$. By Lemma~\ref{2-homogeneous}, there is a 3-generator 5-relator presentation of $A_{p+2}\times T$. More precisely, let
  \begin{equation}
    \label{eq:Jhat-AGL(1,p)^(2)}
    \hat{J} \coloneqq \langle a,b,z \mid a^p = b^{(p-1)/2},\ (a^s)^b = a^{s-1},\ z^3 = (zz^a)^2 = 1,\ z^b = z \rangle.
  \end{equation}
  Let $\Omega\coloneqq\F_p\cup\{\star,\bullet\}$. Define $\hat{\varphi}\colon\{a,b,z\}\to\Alt(\Omega)\times T$ by
  \begin{align}
    \label{eq:Examples 3.4 (3)}
    \begin{split}
      \varphi(a) &\coloneqq (\bar{a},\bar{a}), \\
      \varphi(b) &\coloneqq (\bar{b},\bar{b}), \\
      \varphi(z) &\coloneqq ((0,\star,\bullet),1),
    \end{split}
  \end{align}
  where in the left coordinate $\bar{a},\bar{b}$ are viewed as permutations on $\Omega$ that fix $\star$ and $\bullet$. Lemma~\ref{2-homogeneous} asserts that $\hat{\varphi}$ extends to an isomorphism $\hat{\varphi}\colon\hat{J}\to\Alt(\Omega)\times T$.

  Let $J$ be the quotient of $\hat{J}$ obtained from \eqref{eq:Jhat-AGL(1,p)^(2)} by imposing the relation $(bz^az^{a^{-1}})^{(p+1)/2}=1$. We claim that $J\cong A_{p+2}$. Note $\bar{b}$ is the product of two disjoint $(p-1)/2$-cycles, and $1$ and $p-1$ lie in distinct cycles since $p\equiv3\bmod{4}$. Thus $\varphi(bz^az^{a^{-1}})=(\bar{b}(1,\star,\bullet)(p-1,\star,\bullet),\bar{b})=(\bar{c}_1\bar{c}_2,\bar{b})$, where $\bar{c}_1\bar{c}_2$ are disjoint $(p+1)/2$-cycles. Hence $\varphi((bz^az^{a^{-1}})^{(p+1)/2})=(1,\bar{b}^{(p+1)/2})=(1,\bar{b})$. The normal closure of $\bar{b}$ in $T$ is $T$. Hence, adding the relator $(bz^az^{a^{-1}})^{(p+1)/2}$ to the presentation in \eqref{eq:Jhat-AGL(1,p)^(2)} yields a presentation of $\Alt(\Omega)$. Upon selecting a bijection between $\Omega$ and $\{1,\ldots,p+2\}$, an isomorphism $\Alt(\Omega)\cong A_{p+2}$ is obtained, thus
  \begin{equation}
    \label{eq:p+2-abz}
    J \coloneqq \langle a,b,z \mid a^p=b^{(p-1)/2},\ (a^s)^b = a^{s-1},\ z^3 = (zz^a)^2 = 1,\ z^b = z,\ (bz^az^{a^{-1}})^{(p+1)/2} = 1 \rangle \cong A_{p+2}.
  \end{equation}
  In view of the above argument and \eqref{eq:Examples 3.4 (3)}, an isomorphism $\varphi\colon J\to\Alt(\Omega)$ is defined by $\varphi(a)\coloneqq\bar{a}$, $\varphi(b)\coloneqq\bar{b}$ and $\varphi(z)\coloneqq(0,\star,\bullet)$.
\end{example}

\begin{remark}
  All presentations discussed in Examples~\ref{Examples3.18(1)}, \ref{ex:Examples 3.4 (2)} and \ref{ex:Examples 3.4 (3)} have bit-length $O(\log p)$.
\end{remark}

The following is \cite[Corollary 3.8]{GKKL1}. It reduces the number of generators and relations in the previous example.

\begin{theorem}
  \label{thm:alt-p+2}
  For each prime $p\equiv11\bmod{12}$
  \begin{enumerate}[label=(\roman*)]
  \item $A_{p+2}\times\AGL(1,p)^{(2)}$ has a 2-generator 3-relator presentation with bit-length $O(\log p)$.
  \item $A_{p+2}$ has a 2-generator 4-relator presentation with bit-length $O(\log p)$.\qedhere
  \end{enumerate}
\end{theorem}

\begin{proof}
  Consider the groups $\hat{J}$ (cf. \eqref{eq:Jhat-AGL(1,p)^(2)}) and $J$ (cf. \eqref{eq:p+2-abz}) in Example~\ref{ex:Examples 3.4 (3)}. We make the following modifications to the presentations that define these groups:
  \begin{itemize}
  \item Replace generators $b,z$ with one generator $g$, and \emph{define} words $b\coloneqq g^3$, $z\coloneqq g^{(p-1)/2}$.
  \item Omit the relations $z^3=1$ and $z^b=z$.
  \end{itemize}
  We claim these modifications define the same groups. Let $\hat{J}'$ and $J'$ be the presented groups constructed from $\hat{J}$ and $J$ respectively by the above modifications. We begin by defining $\hat{g}\in\Alt(\Omega)\times T$ such that $\hat{g}^3=\hat{\varphi}(b)$ and $\hat{g}^{(p-1)/2}=\hat{\varphi}(z)$, where $\Omega=\F_p\cup\{\star,\bullet\}$, $T=\AGL(1,p)^{(2)}$, and $\hat{\varphi}\colon\hat{J}\to\Alt(\Omega)\times T$ (cf. \eqref{eq:Examples 3.4 (3)}) are as in Example~\ref{ex:Examples 3.4 (3)}.

  Let $\F_p^{*2}=\langle r \rangle$ as in Example~\ref{ex:Examples 3.4 (3)} and define $\bar{b}$ as in \eqref{eq:AGL-gens}. Note that $(p-1,3)=1$, so there exists $\alpha\in\F_p^{*2}$ such that $\alpha^3=r\bmod{p}$. Let $\bar{g}\in T$ be the permutation which maps $x\mapsto\alpha x$. Define $\hat{g}\coloneqq(\bar{g}(0,\star,\bullet)^{-1},\bar{g})\in\Alt(\Omega)\times T$, where $\bar{g}$ is viewed as fixing $\star$ and $\bullet$ in the first coordinate. Since $p\equiv2\bmod{3}$,
  \[ \hat{g}^3 = (\bar{b},\bar{b}) = \hat{\varphi}(b), \quad \hat{g}^{(p-1)/2} = ((0,\star,\bullet),1) = \hat{\varphi}(z). \]
  Hence there is a surjection $\hat{J}'\to\Alt(\Omega)\times T$, where $a\mapsto\hat{\varphi}(a)$, $g\mapsto\hat{g}$, $b=g^3\mapsto\hat{\varphi}(b)$ and $z=g^{(p-1)/2}\mapsto\hat{\varphi}(z)$. By projecting onto the second coordinate and applying Lemma~\ref{lem:2.1}, we may identify the subgroup $\langle a,b \rangle=\langle a,g^3 \rangle$ of $\hat{J}'$ with $T$ (cf. \eqref{eq:AGL(1,p)^(2)}). Thus $b=g^3$ has order $(p-1)/2$, hence $z^3=1$ holds. Since $b=g^3$ and $z=g^{(p-1)/2}$ commute, the relation $z^b=z$ holds. By Lemma~\ref{lem:2.1} we may identify the subgroup $\langle a,b,z \rangle=\langle a, g^3, g^{(p-1)/2} \rangle$ of $\hat{J}'$ with $A_{p+2}\times T$. Since $(3,p-1)=1$, it follows that $\langle a,b,z \rangle=\langle a,g \rangle=\hat{J}'$. Thus $\hat{J}'\cong A_{p+2}\times T$.

  As in Example~\ref{ex:Examples 3.4 (3)}, we obtain $J'$ as a quotient of $\hat{J}'$ by adding a relator whose normal closure in $\hat{J}'$ corresponds to $1\times T$. Thus $J'\cong A_{p+2}$.
\end{proof}

We more explicitly describe these presentations in Theorem~\ref{thm:p+2}. First we outline how similar presentations can be constructed for symmetric groups.

The following is \cite[Lemma 3.12]{GKKL1}. It is an analogue of Lemma~\ref{2-homogeneous} for symmetric groups.
\begin{lemma}
  \label{lem:2-transitive}
  Let $n\geq3$ be an integer, let $T$ be a group with a \emph{2-transitive} action on $\{1,\ldots,n\}$, and let $\bar{}\colon T\twoheadrightarrow\bar{T}\leq S_n$ be the induced surjection. Suppose $T=\langle X\mid R \rangle$, and $T_1=\langle X_1 \rangle$ is the stabilizer of 1, viewing the elements of $X_1$ as words in $X$. Let $w$ be a word in $X$ that moves 1, when viewed as an element of $T$. (In contrast to Lemma~\ref{2-homogeneous}, we do not require that $\bar{w}\in A_n$.) View $S_n$ as a subgroup of $S_{n+2}$ fixing the points $n+1$ and $n+2$. For every $\sigma\in S_n$ write $\sgn(\sigma)=(-1)^{\epsilon(\sigma)}$, where $\epsilon(\sigma)\in\{0,1\}$. Assume that $\bar{T}\not\leq A_n$. Let $T^+$ be the subgroup of index two in $T$ with $\bar{T}^+=\bar{T}\cap A_n$. Define
  \begin{equation}
    \label{eq:2-transitive}
    J \coloneqq \langle X, z \mid R,\ z^3 = 1,\ (zz^w)^2 = 1,\ [z,X_1] = 1 \rangle.
\end{equation}

Then $J$ is isomorphic to a subgroup of index two in $S_{n+2}\times T$ that projects onto each factor, via the mapping $\varphi\colon X\cup\{z\}\to S_{n+2}\times T$ defined by
\begin{align}
    \varphi(x) &\coloneqq (\bar{x},x) \quad \text{for } x\in X, \\
    \varphi(z) &\coloneqq ((1,n+1,n+2),1).\qedhere
\end{align}
\end{lemma}

\begin{proof}
  View $S_{n+2}\times T$ acting on the disjoint union $\{1,\ldots,n+2\}\sqcup\{1,\ldots,n\}$, and let $H$ be its subgroup of index two that induces a subgroup of $A_{2n+2}$. The image of $\varphi$ satisfies the relations in \eqref{eq:2-transitive}, hence $\varphi$ extends to a homomorphism $\varphi\colon J\to S_{n+2}\times T$. Observe that $\varphi(z)$ and $\varphi(x)$ act as even permutations for each $x\in X$, thus $\varphi(J)\leq H$. We claim that $\varphi(J)=H$. Recall that $T$ acts transitively on $\{1,\ldots,n\}$. Thus $\varphi(z)^{\langle \varphi(X) \rangle}=\{((i,n+1,n+2),1) : 1\leq i\leq n\}$, hence $\langle \varphi(z)^{\langle \varphi(X) \rangle} \rangle=A_{n+2}\times 1$. Therefore $\varphi(J)$ also contains $1\times T^+$.

  Recall that $\bar{T}\not\leq A_n$, hence $\varphi(x)\in H\setminus(A_{n+2}\times T^+)$ for some $x\in X$. The subgroup $A_{n+2}\times T^+$ has index four in $S_{n+2}\times T$, and hence index two in $H$. Therefore $\varphi$ is surjective. Define the surjection $\pi\colon J\to S_{n+2}$ by projecting onto the first coordinate.

  Projecting onto the second coordinate yields an isomorphism $\langle \varphi(X) \rangle\cong T$. Hence by Lemma~\ref{lem:2.1}, we may identify the subgroup $\langle X \rangle$ of $J$ with $T$. The relations $[z,X_1]=1$ ensure that $z^{T_1}=z$. By the orbit-stabilizer theorem, $[T:T_1]=n$, hence $|z^T|\leq n$. Since $\pi(z^T)=\{(m,n+1,n+2) : 1\leq m\leq n\}$ has $n$ elements, $|z^T|=n$. By Lemma~\ref{lem:same-action}, $T$ acts on $z^T$ as it does on $\{1,\ldots,n\}$. The action is 2-transitive, so the relations $z^3=1$ and $(zz^w)^2=1$ imply that $N\coloneqq\langle z^T \rangle\cong A_{n+2}$, by \eqref{eq:carmichael}.

  Clearly $N\unlhd J$ and $J/N=J/\langle z^T \rangle\cong T$. Thus $|J|=|A_{n+2}|\cdot|T|=|H|$, so $\varphi$ is an isomorphism.
\end{proof}

The following example discusses presentations of $S_{p+2}$ and an index two subgroup of $S_{p+2}\times\AGL(1,p)$, as described in \cite[Corollary 3.13]{GKKL1}.
\begin{example}
  \label{ex:Corollary 3.13 (i)}
  Let $p$ be an odd prime, and let $T\coloneqq\AGL(1,p)$. Define $\Omega\coloneqq\F_p\cup\{\star,\bullet\}$, and let $H$ be the index two subgroup of $\Sym(\Omega)\times T$ that induces a subgroup of $\Alt(\Omega\sqcup\F_p)$. Recall the 2-generator 2-relator presentation of $T$ discussed in Example~\ref{ex:Examples 3.4 (2)} (cf. \eqref{eq:AGL(1,p)}). By Lemma~\ref{lem:2-transitive}, there is a 3-generator 5-relator presentation of $H$. More precisely, let $\F_p^*=\langle r \rangle$ and $s(r-1)\equiv-1\bmod{p}$ as in Example~\ref{ex:Examples 3.4 (2)}. Define
  \begin{equation}
    \label{eq:Jhat-AGL-Sym(p+2)}
    \hat{J} \coloneqq \langle a,b,z \mid a^p = b^{p-1},\ (a^s)^b = a^{s-1},\ z^3 = (zz^a)^2 = [z,b] = 1 \rangle.
\end{equation}
  Define $\hat{\varphi}\colon\{a,b,z\}\to H$ by
  \begin{align}
    \label{eq:Corollary 3.13}
    \begin{split}
      \hat{\varphi}(a) &\coloneqq (\bar{a},\bar{a}), \\
      \hat{\varphi}(b) &\coloneqq (\bar{b},\bar{b}), \\
      \hat{\varphi}(z) &\coloneqq ((0,\star,\bullet),1),
    \end{split}
  \end{align}
  where $\bar{a},\bar{b}$ are the permutations defined in \eqref{eq:AGL-gens}, fixing $\star$ and $\bullet$ in the first coordinate. Lemma~\ref{lem:2-transitive} asserts that $\hat{\varphi}$ extends to an isomorphism $\hat{\varphi}\colon\hat{J}\to H$.

  Let $J$ be the quotient of $\hat{J}$ obtained from \eqref{eq:Jhat-AGL-Sym(p+2)} by imposing the relation $(b^2z^az^{a^r})^{(p+1)/2}=1$. We claim that $J\cong S_{p+2}$. Observe that $\varphi(b^2z^az^{a^r})=(\bar{b}^2(1,\star,\bullet)(r,\star,\bullet),\bar{b}^2)=(\bar{c}_1\bar{c}_2,\bar{b}^2)$ for disjoint $(p+1)/2$-cycles $\bar{c}_1,\bar{c}_2$. Hence $\varphi((b^2z^az^{a^r})^{(p+1)/2})=(1,\bar{b}^{p+1})=(1,\bar{b}^2)$. The normal closure of $\bar{b}^2$ in $T$ is $T\cap\Alt(\F_p)$. Hence, adding the relator $(b^2z^az^{a^r})^{(p+1)/2}$ to the presentation in \eqref{eq:Jhat-AGL-Sym(p+2)} yields a presentation of $\Sym(\Omega)$. Upon selecting an bijection between $\Omega$ and $\{1,\ldots,p+2\}$, an isomorphism $\Sym(\Omega)\cong S_{p+2}$ is obtained, thus
  \begin{equation}
    J \coloneqq \langle a,b,z \mid a^p=b^{p-1},\ (a^s)^b = a^{s-1},\ z^3 = (zz^a)^2 = [z,b] = 1,\ (b^2z^az^{a^r})^{(p+1)/2} = 1 \rangle \cong S_{p+2}.
  \end{equation}
  In view of the above argument and \eqref{eq:Corollary 3.13}, an isomorphism $\varphi\colon J\to\Sym(\Omega)$ is defined by $\varphi(a)\coloneqq\bar{a}$, $\varphi(b)\coloneqq\bar{b}$ and $\varphi(z)\coloneqq(0,\star,\bullet)$.
\end{example}

The following is \cite[Corollary 3.13 (ii) and (iii)]{GKKL1}, an analogue of Theorem~\ref{thm:alt-p+2} for symmetric groups.

\begin{theorem}
  \label{thm:sym-p+2}
  Let $p\equiv2\bmod{3}$ where $p>3$ is prime.
  \begin{enumerate}[label=(\roman*)]
  \item The index two subgroup of $S_{p+2}\times\AGL(1,p)$ which induces a subgroup of $A_{2p+2}$ has a 2-generator 3-relator presentation with bit-length $O(\log p)$.
  \item $S_{p+2}$ has a 2-generator 4-relator presentation with bit-length $O(\log p)$.\qedhere
  \end{enumerate}
\end{theorem}

\begin{proof}
  As in the proof of Theorem~\ref{thm:alt-p+2} we modify the presentations $J$ and $\hat{J}$ in Example~\ref{ex:Corollary 3.13 (i)} by rewriting generators $b$ and $z$ in terms of a single generator $g$, where $b\coloneqq g^3$ and $z\coloneqq g^{p-1}$, and omitting the relations $z^3=1$, $[z,b]=1$. A proof almost identical to that of Theorem~\ref{thm:alt-p+2} shows that these presentations are correct.
\end{proof}

The following provides more explicit descriptions of the presentations described in Theorems~\ref{thm:alt-p+2} and \ref{thm:sym-p+2}.

\begin{theorem}
  \label{thm:p+2}
  If $p\equiv11\bmod{12}$ is prime and $G\coloneqq A_{p+2}$, or if $p\equiv2\bmod{3}$ where $p>2$ is prime and $G\coloneqq S_{p+2}$, then $G$ has a 2-generator 4-relator presentation with bit-length $O(\log p)$. More precisely,
  \begin{equation}
    \label{eq:p+2}
    J \coloneqq \langle a, g \mid a^p = b^\kappa,\ (a^s)^b = a^{s-1},\ (zz^a)^2=1,\ h=1 \rangle \cong G,
\end{equation}
  where one of the following holds:
  \begin{itemize}
  \item $G\coloneqq A_{p+2}$, $p\equiv11\bmod{12}$ is prime,
    $\F_p^{*2}=\langle r \rangle$,
    $s(r-1)\equiv-1\bmod{p}$, $\kappa\coloneqq(p-1)/2$, $b\coloneqq g^3$,
    $z\coloneqq g^\kappa$, and $h\coloneqq(bz^az^{a^{-1}})^{(p+1)/2}$ (cf. Theorem~\ref{thm:alt-p+2}(ii)).
  \item $G\coloneqq S_{p+2}$, $p\equiv2\bmod{3}$ is an odd prime, $\F_p^*=\langle r \rangle$, $s(r-1)\equiv-1\bmod{p}$, $\kappa\coloneqq p-1$, $b\coloneqq g^3$, $z\coloneqq g^\kappa$, and $h\coloneqq(b^2z^az^{a^r})^{(p+1)/2}$ (cf. Theorem~\ref{thm:sym-p+2}(ii)).
  \end{itemize}
  These conditions are summarized in Table~\ref{tab:p+2}.
  \begin{table}[h]
    \centering
    \caption{Various definitions for presentations of $A_{p+2}$ and $S_{p+2}$.}
    \label{tab:p+2}
    \bgroup
    \def\arraystretch{1.4}
    \begin{tabular}{|c|c|}
      \hline
      $A_{p+2}$ & $S_{p+2}$ \\
      \hline
      \hline
      $p\equiv11\bmod{12}$ is prime & $p\equiv2\bmod{3}$, $p>2$ is prime \\
      \hline
      $\F_p^{*2}=\langle r \rangle$ & $\F_p^*=\langle r \rangle$ \\
      \hline
      \multicolumn{2}{|c|}{$s(r-1)\equiv-1\bmod{p}$} \\
      \hline
      $\kappa\coloneqq(p-1)/2$ & $\kappa\coloneqq p-1$ \\
      \hline
      \multicolumn{2}{|c|}{$b\coloneqq g^3$} \\
      \multicolumn{2}{|c|}{$z\coloneqq g^\kappa$} \\
      \hline
      $h\coloneqq(bz^az^{a^{-1}})^{(p+1)/2}$ & $h\coloneqq(b^2z^az^{a^r})^{(p+1)/2}$ \\
      \hline
    \end{tabular}
    \egroup
  \end{table}
  
  Furthermore, define $T$ by
  \begin{equation}
    \label{eq:p+2-T}
    T \coloneqq
    \begin{cases}
      \AGL(1,p)^{(2)} & \text{if } G = A_{p+2}, \\
      \AGL(1,p) & \text{if } G = S_{p+2}.
    \end{cases}
  \end{equation}
  Let $\hat{G}$ be the unique index two subgroup of $S_{p+2}\times T$, which induces a subgroup of $A_{2p+2}$ when we view $S_{p+2}\times T$ acting on the disjoint union $\{1\,\ldots,p+2\}\sqcup\{1,\ldots,p\}$. For example, if $G=A_{p+2}$, then $\hat{G}=G\times T$. Then $\hat{G}$ has a 2-generator 3-relator presentation with bit-length $O(\log p)$. More precisely, omit the relation $h=1$ from $\eqref{eq:p+2}$, and let $\hat{J}$ be the resulting group. Then $\hat{J}\cong\hat{G}$ (cf. Theorems~\ref{thm:alt-p+2}(i) and \ref{thm:sym-p+2}(i)).
\end{theorem}

\begin{remark}
  \label{rem:p+2-isomorphisms}
  Adopt the notation of Theorem~\ref{thm:p+2}. Explicit isomorphisms $J\cong G$ and $\hat{J}\cong\hat{G}$ may be defined as follows. View $G$ as acting on $\F_p\cup\{p+1,p+2\}$, using the representatives $\{1,2,\ldots,p\}$ of $\F_p$. Let $\alpha\in\F_p^*$ be such that $\alpha^3=r$; such an $\alpha$ exists because $(3,p-1)=1$. Define $\hat{\varphi}\colon\{a,g\}\to\hat{G}$ by
  \begin{align}
    \label{eq:p+2-phi-hat}
    \begin{split}
      \hat{\varphi}(a) &\coloneqq ((1,2,\ldots,p),(1,2,\ldots,p)), \\
      \hat{\varphi}(g) &\coloneqq ((x\mapsto\alpha x : x\in\F_p^*)(p,p+1,p+2)^\kappa,(x\mapsto\alpha x : x\in\F_p^*)).
    \end{split}
  \end{align}
  Let $\varphi\colon\{a,g\}\to G$ be the composition of $\hat{\varphi}\colon\{a,g\}\to\hat{G}$ with the projection $\hat{G}\twoheadrightarrow G$ onto the first coordinate. More explicitly, define
  \begin{align}
    \label{eq:p+2-phi}
    \begin{split}
      \varphi(a) &\coloneqq (1,2,\ldots,p), \\
      \varphi(g) &\coloneqq (x\mapsto\alpha x : x\in\F_p^*)(p,p+1,p+2)^\kappa.
    \end{split}
  \end{align}
  Then $\varphi$ and $\hat{\varphi}$ extend to isomorphisms $\varphi\colon J\to G$ and $\hat{\varphi}\colon\hat{J}\to\hat{G}$ (cf. Theorems~\ref{thm:alt-p+2} and \ref{thm:sym-p+2}).
  A calculation shows that $\bar{b}\coloneqq\varphi(b)=(x\mapsto rx:x\in\F_p^*)$, $\varphi(z)=(p,p+1,p+2)$ and
  \[
    \hat{\varphi}(h) =
    \begin{cases}
      (1,\bar{b}) & \text{if $G=A_{p+2}$}, \\
      (1,\bar{b}^2) & \text{if $G=S_{p+2}$.}
    \end{cases}
  \]
  In particular, 3 does not divide the order of $\hat{\varphi}(h)$, because $(3,p-1)=1$.
\end{remark}

\begin{remark}
  \label{rem:compatible}
  Adopt the notation of Theorem~\ref{thm:p+2} and Remark~\ref{rem:p+2-isomorphisms}. Let $N$ be the normal closure of $\hat{\varphi}(h)$ in $\hat{G}$. Then $T=1\times (T\cap A_p)$ (note that $T=T\cap A_p$ if $G=A_{p+2}$). Thus $N$ is the kernel of the surjection $\pi_G\colon\hat{G}\twoheadrightarrow G$ given by projecting onto the first coordinate. Hence the following diagram commutes:
  \begin{center}
    \begin{tikzcd}
      \hat{J} \arrow[dd, "\pi_J", two heads] \arrow[rr, "\hat{\varphi}", two heads, hook] &  & \hat{G} \arrow[dd, "\pi_G", two heads] \\
      &  &                    \\
      J \arrow[rr, "\varphi", two heads, hook]                        &  & G                 
    \end{tikzcd}
  \end{center}
  where $\pi_J\colon\hat{J}\to J$ is the surjection induced by imposing the relation $h=1$.
\end{remark}

In what follows, we require a special case of a group-theoretic version of ``Horner's rule'' (cf. \cite[p.\ 512]{BKL}). For every two elements $v,f$ in a group, and every positive integer $n$
\begin{equation}
  \label{eq:horner}
  vv^fv^{f^2}\cdots v^{f^n} = (vf^{-1})^nvf^n.
\end{equation}

The following is \cite[Remarks 3.11 and 3.16]{GKKL1}, and is important for later results. In particular it is used to `glue' the presentations of Theorem~\ref{thm:p+2}.

\begin{lemma}
  \label{rem:p+2-cycle-bit-length}
  As in Theorem~\ref{thm:p+2}, suppose that either $G=A_{p+2}$ with $p\equiv11\bmod{12}$ prime, or $G=S_{p+2}$ with $p\equiv2\bmod{3}$ where $p>2$ is prime. Let $\varphi\colon J\to G$ be the isomorphism defined in Remark~\ref{rem:p+2-isomorphisms}. In $G$, even permutations of bounded support and cycles $(i,i+1,\ldots,j)$ with $j-i$ even have bit-length $O(\log p)$ in $\{\varphi(a),\varphi(g)\}$. If $p\equiv11\bmod{12}$ and $G=S_{p+2}$, then, in $G$, odd permutations of bounded support and cycles $(i,i+1,\ldots,j)$ with $j-i$ odd also have bit-length $O(\log p)$ in $\{\varphi(a),\varphi(g)\}$.
\end{lemma}
\begin{proof}  
  The isomorphism $\varphi$ may be used to identify elements of $J$, which we write as words in $\{a,g\}$, with permutations in $G$. If $w\in J$ is a word in $\{a,g\}$, and $\sigma\in G$, we write $w\equiv\sigma$ or $\sigma\equiv w$ to signify that $\varphi(w)=\sigma$.

  Let $b,z$ be the words in $\{a,g\}$ defined in Theorem~\ref{thm:p+2}; these have bit-length $O(\log p)$ in $\{a,g\}$. Remark~\ref{rem:p+2-isomorphisms} states that $b\equiv(x\mapsto rx : x\in\F_p^*)$ and $z\equiv(p,p+1,p+2)$. Thus the 3-cycles $(i,p+1,p+2)\equiv z^{a^i}$ for $1\leq i\leq p$ have bit-length $O(\log p)$ in $\{\varphi(a),\varphi(g)\}$, hence so too do all permutations of bounded support.

  Consider a cycle $(i,i+1,\ldots,j)$ with $i<j$. First assume that $j\leq p-1$. Let $x\coloneqq za^{-1}z^{-1}az\equiv(1,p)(p+1,p+2)$. Then $ax\equiv(1,\ldots,p-1)(p+1,p+2)$. Thus $a^{-j}(ax)^j\equiv(p,1,\ldots,j)(p+1,p+2)^j$ and
  \begin{equation}
    \label{eq:p+2-k-l-cycle}
    a^{-j}(ax)^{j-i}a^i=a^{-j}(ax)^j[a^{-i}(ax)^i]^{-1}\equiv(i,i+1,\ldots,j)(p+1,p+2)^{j-i}.
  \end{equation}
  Hence, if $j-i$ is even and $j\leq p-1$, then $(i,i+1,\ldots,j)$ has bit-length $O(\log p)$ in $\{\varphi(a),\varphi(g)\}$.
  Next define words
  \begin{equation}
    \label{eq:p+2-even-cycles}
    c(i,j) \coloneqq
    \begin{cases}
      a^{-j}(ax)^{j-i}a^i & \text{for } 1\leq i \leq j\leq p-1 \text{ (cf. \eqref{eq:p+2-k-l-cycle})}, \\
      z^{(za)^{-2}}c(i,p-2) & \text{for } 1\leq i\leq p-2, \text{ and } j=p, \\
      z^{(za)^{-1}}c(i,p-1) & \text{for } 1\leq i\leq p-1, \text{ and } j=p+1, \\
      zc(i,p) & \text{for } 1\leq i\leq p-2, \text{ and } j=p+2.
    \end{cases}
  \end{equation}
  Then $c(i,j)\equiv(i,i+1,\ldots,j)(p+1,p+2)^{j-i}$ wherever this is defined. In particular every cycle $(i,i+1,\ldots,j)$ with $i<j$ and $j-i$ even has bit-length $O(\log p)$ in $\{\varphi(a),\varphi(g)\}$ (recall that $z\equiv(p,p+1,p+2)$).

  Now suppose $G=S_{p+2}$ and $p\equiv11\bmod{12}$. It suffices to show that some transposition, such as $(p+1,p+2)$, has bit-length $O(\log p)$ in $\{\varphi(a),\varphi(g)\}$. Define words
  \begin{align}
    \label{eq:p+2-transposition-explicit}
    \begin{split}
      b_2 &\coloneqq b^{(p-1)/2} \equiv (1,p-1)(2,p-2)\cdots\big(\frac{p-1}{2},\frac{p+1}{2}\big), \\
      c_\bullet &\coloneqq c\big(1,\frac{p-1}{2}\big)\; c\big(\frac{p+1}{2},p-1\big)^{-1} \equiv \big(1,2,\ldots,\frac{p-1}{2}\big)\big(p-1,p-2,\ldots,\frac{p+1}{2}\big), \\
      d &\coloneqq z^a (z^{a^{-1}})^{-1} z^a \equiv (1,p-1)(p+1,p+2), \\
      v &\coloneqq (dc_\bullet)^{(p-1)/2} \equiv (1,p-1)(2,p-2)\cdots\big(\frac{p-1}{2},\frac{p+1}{2}\big)(p+1,p+2), \\
      t &\coloneqq vb_2 \equiv (p+1,p+2).
    \end{split}
  \end{align}
  That $v\equiv(1,p-1)(2,p-2)\cdots(\frac{p-1}{2},\frac{p+1}{2})(p+1,p+2)$ follows from \eqref{eq:horner}. Specifically, note that $c_\bullet$ has order $(p-1)/2$ and $d$ has order 2 in $J$. Hence
  \[ v = (dc_\bullet)^{(p-1)/2} = (dc_\bullet)^{(p-1)/2}dc_\bullet^{(p-1)/2}d = dd^{c_\bullet^{-1}}d^{c_\bullet^{-2}}\cdots d^{c_\bullet^{-(p-1)/2}}d \]
  in $J$. Observe that $d^{c_\bullet^{-i}}\equiv(\frac{p+1}{2}-i,\frac{p-1}{2}+i)(p+1,p+2)$ for all $1\leq i\leq(p-1)/2$. Since $(p-1)/2$ is odd, $v$ corresponds to the permutation indicated in \eqref{eq:p+2-transposition-explicit}.

  Multiplication by $(p+1,p+2)$ changes the support of a permutation by at most 2, hence odd permutations of bounded support have bit-length $O(\log p)$ in $\{\varphi(a),\varphi(g)\}$. This includes all odd cycles $(i,i+1,\ldots,j)$ with $0<j-i\leq4$. For the cycles $(i,i+1,\ldots,j)$ with $j-i>4$ odd, recall that $(i,i+1,\ldots,j)(p+1,p+2)^{j-i}$ has bit-length $O(\log p)$. Post multiplication by $(p+1,p+2)$ yields the desired result.
\end{proof}

\begin{remark}
  The words constructed in Lemma~\ref{rem:p+2-cycle-bit-length} differ from those of \cite{GKKL1}, primarily due to the order of multiplication of permutations. However our definition of $v$ differs from that of \cite[Remark 3.16]{GKKL1} for another reason. The construction of $v$ in the latter corresponds to $(2,p-2)\cdots(\frac{p-1}{2},\frac{p+1}{2})$. From this a word is constructed, which is claimed to correspond to $(p+1,p+2)$, but actually corresponds to $(1,p-1)$. This error affects the correctness of the explicit presentation of the symmetric group given in \cite[Section 3.5]{GKKL1}.
\end{remark}

\subsection{Gluing presentations}

The following is a slightly special case of \cite[Corollary 3.28]{GKKL1}, and follows from more general statements  \cite[Lemma 3.23, Lemma 3.27]{GKKL1} about `gluing' presentations of two alternating groups or two symmetric groups.
\begin{theorem}
  \label{thm:glue}
  If $A_m$ has a presentation with $M$ relations and $6\leq k\leq m-2$, then $A_{2m-k}$ has a presentation with $M+4$ relations, and an additional generator. The same result holds for symmetric groups under the weaker hypothesis that $6\leq k\leq m-1$.

  Such presentations may be constructed as follows. We will describe the construction in terms of alternating groups only, but the same construction works when they are replaced by symmetric groups throughout.

  Let $J\coloneqq\langle X\mid R \rangle\cong A_m$. View $A_{2m-k}$ as $\Alt(\{-m+k+1,\ldots,m\})$, and let $\pi\colon J\hookrightarrow A_{2m-k}$ be an embedding with $\pi(J)=\Alt(\{1,\ldots,m\})$. Let $\bar{y}\in A_{2m-k}$ be a permutation sending $\{-m+k+1,\ldots,k\}$ to $\{1,\ldots,m\}$ that induces the identity on $\{1,\ldots,k\}$ (the conditions $k\leq m-2$ and $k\leq m-1$ respectively imply the existence of $\bar{y}$). Let $a,b,c,d,e$ be words in $X$ such that
  \begin{gather*}
    \pi(a)=(1,2,3), \quad \pi(b)\in\Alt(\{4,\ldots,k\}), \quad \pi(c)\in\Alt(\{1,\ldots,k\}), \\
    \pi(d)\in\Alt(\{4,\ldots,m\}), \quad \pi(e)^{\bar{y}}\in\Alt(\{-m+k+1,\ldots,3\}), \\
    \langle \pi(a), \pi(c) \rangle = \Alt(\{1,\ldots,k\}), \quad \langle \pi(b), \pi(d) \rangle = \Alt(\{4,\ldots,n\}), \\
    \langle \pi(a), \pi(e)^{\bar{y}} \rangle = \Alt(\{-m+k+1,\ldots,3\}).
  \end{gather*}
  The conditions on $m,k$ guarantee that such words exist. Define $\theta\colon F_{X\cup\{y\}}\to A_{2m-k}$ by $\theta|_X\coloneqq\pi|_X$ and $\theta(y)\coloneqq\bar{y}$. Since $A_{2m-k}=\langle \pi(X)\cup\pi(X)^{\bar{y}} \rangle$, there exists a word $w$ in $X\cup X^y$ such that $\theta(w)=\bar{y}$. If $w$ is such a word, then
  \[ J'\coloneqq\langle X, y \mid R,\ a=a^y,\ c=c^y,\ [d,e^y]=1,\ y=w \rangle\cong A_{2m-k}. \]
  The map $\theta$ factors through $J'$, inducing an isomorphism $J'\cong A_{2m-k}$. Observe $b$ is not directly used in the presentation, however we require its existence in the proof.
\end{theorem}

\begin{proof}
  The words $a,c,d,e,\bar{y}$ were chosen so that the defining relations are satisfied; hence $\theta$ factors through $J'$, inducing an epimorphism $\varphi\colon J'\to A_{2m-k}$. This maps $\langle X \rangle$ in $J'$ onto the subgroup $\Alt(\{1,\ldots,m\})$ of $A_{2m-k}$. By Lemma~\ref{lem:2.1}, $\varphi$ maps $\langle X \rangle$ isomorphically onto $\Alt(\{1,\ldots,m\})$, and hence also maps $\langle X^y \rangle$ isomorphically onto $\Alt(\{1,\ldots,m\})^{\tilde{y}}=\Alt(\{-m+k+1,\ldots,k\})$. Identify $\langle X \rangle$ and $\langle X^y \rangle$ with these subgroups of $A_{2m-k}$.

  Observe that $a,b,c\in\langle X \rangle$ in $J'$ and $\varphi(b)\in\langle \varphi(a), \varphi(c) \rangle$. Hence $b\in\langle a,c \rangle$. Since the relations $a=a^y$ and $c=c^y$ hold in $J'$, it follows that $b=b^y$. Since $d,a\in\langle X \rangle$ and their images under $\varphi$ have disjoint support, $[d,a]=1$. Similarly $[b,a]=1$ ($b,a\in\langle X \rangle$) and $[b,e^y]=1$ ($b=b^y,e^y\in\langle X^y \rangle$). Combine this with the relation $[d,e^y]=1$ to deduce
  \begin{equation}
    \label{eq:subgroups-commute}
    [\langle b,d \rangle,\langle a,e^y \rangle] = 1.
  \end{equation}
  Note that $\varphi$ maps $\langle b,d \rangle$ isomorphically onto $\Alt(\{4,\ldots,m\})$, and $\langle a,e^y \rangle$ isomorphically onto $\Alt(\{-m+k+1,\ldots,3\})$. Finally, the relation $y=w$ expressing $y$ as a word in $X\cup X^y$ implies that $J'=\langle X\cup X^y \rangle$.

  All the above applies when alternating groups are replaced by symmetric groups throughout. To complete the proof in both cases, we use known presentations of alternating and symmetric groups. In the alternating case, we use Carmichael's presentation~\eqref{eq:carmichael}
  \begin{align}
    \label{eq:carmichaeloo}
    \begin{split}
      A_n = \langle x_i,\ -m+1+k\leq i\leq m,\ i\neq1,2 &\mid x_i^3 = (x_ix_j)^2 = 1, \\
      &\phantom{\mid{}} \text{for all possible $i,j$ with $i\neq j$ } \rangle.
    \end{split}
  \end{align}

  There are unique $x_i\in\langle X \rangle$ with $\varphi(x_i)=(1,2,i)$ for $3\leq i\leq m$, and unique $y_j\in\langle X^y \rangle$ with $\varphi(y_j)=(1,2,j)$ for $-m+1+k\leq j\leq k$, $j\neq1,2$. The $x_i$ and the $y_j$ generate $\langle X \rangle$ and $\langle X^y \rangle$ respectively. By uniqueness, $x_i=y_i$ when $3\leq i\leq k$, since both lie in $\langle a,c \rangle=\langle a^y,c^y \rangle\leq\langle X \rangle\cap\langle X^y \rangle$. Hence it makes sense to define $x_i\coloneqq y_i$ for all $-m+1+k\leq i\leq k$ with $i\neq1,2$.

  The $x_i$ generate $J'=\langle X\cup X^y \rangle$. The relations in \eqref{eq:carmichaeloo} are clearly satisfied, with the possible exception of
  \[ (x_jx_i)^2 = (x_ix_j)^2 = 1 \text{ when } k < i \leq n \text{ and } -m+1+k\leq j\leq 0. \]
  Let $g\in\langle b,d\rangle \cong\Alt\{4,\ldots,m\}$ be such that $\varphi(g)$ sends $i$ to $k$. Then $g$ commutes with $x_j\in\langle a,e^y \rangle$ by \eqref{eq:subgroups-commute}, and $x_i^g=x_k$ since $\varphi(g)$ fixes 1 and 2. Thus
  \[ [(x_ix_j)^2]^g = (x_kx_j)^2 = 1. \]
  Hence the relations in \eqref{eq:carmichaeloo} hold. Thus $J'$ is an epimorphic image of $A_{2m-k}$, hence $\varphi\colon J'\to A_{2m-k}$ is an isomorphism.

  The symmetric group case is simpler; we use Moore's presentation \eqref{eq:moointro}
  \begin{align}
    \label{eq:moo}
    \begin{split}
      S_{2m-k} = \langle x_i, -m+1+k\leq i<m &\mid x_i^2 = (x_{i-1}x_i)^3 = (x_ix_j)^2 = 1, \\
      &\phantom{\mid{}} \text{for all possible $i,j$ with $|i-j|\geq2$}\rangle.
    \end{split}
  \end{align}
  As before, since $\langle a,c \rangle=\langle a^y,c^y \rangle\cong\Alt(\{1\,\ldots,k\})$, there exist $x_i\in J'$ with $\varphi(x_i)=(i,i+1)$ for $-m+k+1\leq i<m$, where $x_i\in\langle X \rangle$ for $1\leq i<m$, and $x_i\in\langle X^y \rangle$ for $-m+k+1\leq i<k$. Furthermore, the $x_i$ generate $J'=\langle X\cup X^y \rangle$. Every two distinct $x_i$ either both lie in $\langle X \rangle$, both lie in $\langle X^y \rangle$, or both commute by \eqref{eq:subgroups-commute} since one lies in $\langle b,d \rangle$ and the other lies in $\langle a,e^y \rangle$. From this it is clear that the relations in \eqref{eq:moo} are satisfied. Thus $J'$ is an epimorphic image of $S_{2m-k}$, hence $\varphi\colon J'\to S_{2m-k}$ is an isomorphism.
\end{proof}

\begin{remark}
  \label{rem:glue}
  Theorem~\ref{thm:glue} is a special case of \cite[Corollary 3.28]{GKKL1}. The proof of the latter has a error. There the hypothesis $k\leq m-2$ is weakened to $k\leq m-1$ in the alternating group case. In the proof $y$ must also correspond to a permutation sending $\{-m+k+1,\ldots,k\}$ to $\{1,\ldots,m\}$ which fixes $\{1,\ldots,k\}$ pointwise. However, if $k=m-1$, then the only such permutation is the transposition $(0,m)$, which does not lie in the alternating group.
\end{remark}

The following is a stronger version of Betrand's postulate. It is used in Theorem~\ref{thm:alt-sym-n} to show that the presentations in Theorem~\ref{thm:p+2} can be glued together (using Theorem~\ref{thm:glue}) to form presentations of alternating and symmetric groups of sufficiently large degree.

\begin{lemma}
  \label{lem:Mor}
  Let $n$ be an integer such that either $14\leq n\leq 20$, $26\leq n\leq 44$ or $n\geq50$. Then there is a prime $p\equiv11\bmod{12}$ such that $(n+2)/2\leq p\leq n-3$. If $n\notin\{14,26,50\}$, then we can ensure that $p$ lies in the range $(n+2)/2\leq p\leq n-4$.
\end{lemma}

\begin{proof}
  By \cite{Mor}, for every $x\geq25$ and every integer $a$ with $\gcd(a,12)=1$, there exists a prime $p\equiv a\bmod{12}$ in the interval $(x,1.94x)$. In particular there exists a prime $p\equiv11\bmod{12}$ in all such intervals. If $n$ is an integer, then
  \[ n-4=1.94\big(\frac{n+2}{2}\big)+0.03(n-198). \]
  Thus if $n\geq198$, then $1.94(n+2)/2\leq n-4$ and $(n+2)/2\geq100>25$. Hence there is a prime $p\equiv11\bmod{12}$ such that $(n+2)/2\leq p\leq n-4$. For the remaining integers $n$, note that
  \[ \frac{n+2}{2}\leq p\leq n-4 \iff p+4\leq n\leq 2p-2. \]
  Thus the primes $p\in\{47,59,107\}$ cover the ranges $51\leq n\leq 96$, $63\leq n\leq 114$ and $111\leq n\leq 210$, hence for all $n\geq51$ there exists such a prime. The ranges $15\leq n\leq 20$ and $27\leq n\leq 44$ are covered by $p=11$ and $p=23$, respectively. Finally, if $n\in\{14,26,50\}$, then $p\coloneqq n-3\equiv11\bmod{12}$ is prime, and clearly $(n+2)/2\leq p\leq n-3$.
\end{proof}

\begin{theorem}
  \label{thm:alt-sym-n}
  Let $n$ be an integer such that either $13\leq n\leq20$, $25\leq n\leq 44$ or $n\geq49$. Both $A_n$ and $S_n$ have 3-generator 7-relator presentations with bit-length $O(\log n)$.
\end{theorem}

\begin{proof}
  The groups $A_{13},A_{25},A_{49}$ and $S_{13},S_{25},S_{49}$ are handled by Theorem~\ref{thm:p+2} (cf. \eqref{eq:p+2}), and $A_{14},A_{26},A_{50}$ are handled by Example~\ref{Examples3.18(1)} (cf. \eqref{eq:alt-p+3}).

  Let $\bar{G}$ be $A_n$ or $S_n$ for one of the remaining cases. By Lemma~\ref{lem:Mor}, there is a prime $p\equiv11\bmod{12}$ such that $(n+2)/2\leq p\leq n-3$, and if $\bar{G}=A_n$ we can ensure that $p\leq n-4$. Let $G\coloneqq A_{p+2}$ if $\bar{G}=A_n$, and $G\coloneqq S_{p+2}$ if $\bar{G}=S_n$.

  Theorem~\ref{thm:p+2} defines groups $J\coloneqq\langle X\mid R', h=1 \rangle$ and $\hat{J}\coloneqq\langle X\mid R' \rangle$ with $|X|=2$ and $|R'|=3$. A subgroup $\hat{G}$ of $G\times T$, which projects onto $G$ in the first coordinate, is also defined. Remark~\ref{rem:p+2-isomorphisms} defines isomorphisms $\varphi\colon J\to G$ and $\hat{\varphi}\colon\hat{J}\to\hat{G}$. By Remark~\ref{rem:compatible} these isomorphisms are compatible with the maps $\hat{J}\twoheadrightarrow J$ and $\hat{G}\twoheadrightarrow G$ given by adding the relation $h=1$ and projection onto the first coordinate, respectively.

  Let $m\coloneqq p+2$ and $k\coloneqq 2p+4-n$, so $n=2m-k$. The conditions on $p$ imply that $6\leq k\leq m-2$ if $\bar{G}=A_n$, and $6\leq k\leq m-1$ if $\bar{G}=S_n$. By using the presentation of $G$ in Theorem~\ref{thm:p+2}, Theorem~\ref{thm:glue} furnishes a 3-generator 8-relator presentation of $\bar{G}$.

  More precisely, view $\bar{G}$ as acting on $\{-m+k+1,\ldots,m\}=\{-p+k-1,\ldots,p+2\}$, and embed $J\hookrightarrow\bar{G}$ via $\varphi\colon J\to G\leq\bar{G}$. To avoid conflicts in notation with Theorem~\ref{thm:p+2}, we write $\tilde{a}$ instead of $a$ in Theorem~\ref{thm:glue}. If words $\tilde{a},c,d,e$ in $X$, a word $w$ in $X\cup X^y$, and $\bar{y}\in\bar{G}$ all satisfy the conditions of Theorem~\ref{thm:glue}, then
  \[ \langle X, y \mid R',\ h=1,\ \tilde{a}=\tilde{a}^y,\ c=c^y,\ [d,e^y]=1,\ y=w \rangle \cong \bar{G}. \]
  An isomorphism is defined by $x\mapsto\varphi(x)$ and $y\mapsto\bar{y}$. Note that $\varphi(\tilde{a})=(1,2,3)$ is a 3-cycle. Assume further that $\tilde{a}$ is chosen so that $\hat{\varphi}(\tilde{a})=((1,2,3),1)$. Recall that $\hat{\varphi}(h)\in1\times T$ (cf. Remark~\ref{rem:p+2-isomorphisms}), hence $\tilde{a}$ and $h$ commute in $\hat{J}$. We claim that
  \[ \bar{J}\coloneqq \langle X, y \mid R',\ \tilde{a}=(\tilde{a}h)^y,\ c=c^y,\ [d,e^y]=1,\ y=w \rangle \cong \bar{G}. \]
  An isomorphism $\bar{J}\cong\bar{G}$ is defined by $x\mapsto\varphi(x)$ and $y\mapsto\bar{y}$. It suffices to show that $h=1$ in $\bar{J}$. Define $\pi\colon\hat{J}=\langle X\mid R' \rangle\twoheadrightarrow\langle X \rangle\leq\bar{J}$ by $x\mapsto x$ for each $x\in X$. Since $\tilde{a}\in\hat{J}$ has order dividing 3, so too does $\pi(\tilde{a})=\tilde{a}=(\tilde{a}h)^y\in\bar{J}$. Hence $\tilde{a}h\in\bar{J}$ has order dividing 3. Since $\tilde{a}$ and $h$ commute in $\hat{J}$, their images under $\pi$ commute in $\bar{J}$. Thus $h\in\bar{J}$ also has order dividing 3. However by Remark~\ref{rem:p+2-isomorphisms}, the order of $h\in\hat{J}$ is not divisible by 3, hence neither is the order of $\pi(h)=h\in\bar{J}$. It follows that $h=1$ in $\bar{J}$, as required.
  
  We have produced a 3-generator 7-relator presentation of $\bar{G}$. To show that such a presentation exists with bit-length $O(\log n)$, it suffices to show that $\tilde{a},c,d,e,w$ can be chosen to be words of bit-length $O(\log n)$ in the generators. We sketch how this can be done below, however more detail will be given in the next section.

  The words $c,d,e$ can be chosen to represent permutations in $G\cong J$ which can be written as the product of at most 2 cycles of the form $(i,i+1,\ldots,j)$ and a permutation of bounded support. By Lemma~\ref{rem:p+2-cycle-bit-length} we can ensure $c,d,e$ have bit-length $O(\log n)$ in $X$.

  Define $\tilde{a}\coloneqq (z^{a^3})^{z^{a^2}z^a}$, where $a$ refers to the element of $X$ as specified in Theorem~\ref{thm:p+2} with $\varphi(a)=(1,2,\ldots,p)$, and $z$ is the word in $X$ defined in Theorem~\ref{thm:p+2} with $\varphi(z)=(p,p+1,p+2)$. Then $\varphi(\tilde{a})=(1,2,3)$ and $\hat{\varphi}(z)\in G\times1$, hence $\hat{\varphi}(\tilde{a})=((1,2,3),1)$ as required.

  It remains to demonstrate the existence of $\bar{y}\in\bar{G}$ and a word $w$ in $X\cup X^y$ of bit-length $O(\log n)$, satisfying the conditions of Theorem~\ref{thm:glue}. Let
  \begin{equation}
    \label{eq:y-bar}
    \bar{y} \coloneqq
    \begin{cases}
      (-p+k-1,p+2)(-p+k,p+1)\cdots(-1,k+2)(0,k+1) &\text{if $n$ is odd or $\bar{G}=S_n$}, \\
      (-p+k-1,p+2,-p+k,p+1)(-p+k+1,p)\cdots(0,k+1) & \text{if $n$ is even and $\bar{G}=A_n$.}
    \end{cases}
  \end{equation}
  
  Then $\bar{y}$ satisfies the required conditions of Theorem~\ref{thm:glue}. We defer the details of how $w$ may be defined to the explicit presentations given in Theorems~\ref{thm:sym-explicit} and \ref{thm:alt-explicit}.
\end{proof}

\begin{remark}
  \label{rem:Theorem 3.40}
  In \cite[Theorem 3.40]{GKKL1} the existence of such presentations of $A_n$ and $S_n$ is asserted for all $n\geq5$. \cite{GKKL1} discusses how to deal with the remaining values of $n$ not addressed by Theorem~\ref{thm:alt-sym-n} on a case-by-case basis.
\end{remark}

\begin{remark}
  \label{rem:Theorem 3.40-wrong}
  Theorem~\ref{thm:alt-sym-n} and its proof are based on \cite[Theorem 3.40]{GKKL1}. The proof of the latter is flawed because it relies on \cite[Corollary 3.28]{GKKL1}. See Remark~\ref{rem:glue} for a discussion of a flaw in the proof of this corollary. As a result of this flaw, some alternating groups are not properly addressed by \cite[Theorem 3.40]{GKKL1}.

  Specifically, let $n\geq5$ be an integer such that there is a unique prime $p\equiv11\bmod{12}$ with $(n+2)/2\leq p\leq n-3$, namely $p=n-3$. Then \cite[Theorem 3.40]{GKKL1} fails to prove the existence of a 3-generator 7-relator presentation of $A_n$ with bit-length $O(\log n)$.

  Such an integer $n$ is of the form $n=p+3$ for a prime $p\equiv11\bmod{12}$. In \cite[Examples 3.18 (1)]{GKKL1} 3-generator 7-relator presentations of $A_{p+3}$ with bit-length $O(\log p)$ are defined for primes $p>3$. However, their example has several errors. These errors are discussed in Remark~\ref{rem:p+3-errors}, and corrected in Example~\ref{Examples3.18(1)}.
\end{remark}

The following is based on \cite[Remark 3.37]{GKKL1}. It ensures that various constructions in \cite{GKKL1} produce presentations of bit-length $O(\log n)$.

\begin{lemma}
  \label{lem:n-cycle-bit-length}
  Let $\bar{J}=\langle X,y \mid R \rangle$ and $\bar{G}$ be defined as in the proof of Theorem~\ref{thm:alt-sym-n}, and $\bar{\varphi}\colon\bar{J}\to\bar{G}$ be the indicated isomorphism. Permutations of bounded support in $\bar{G}$ and cycles $(i,i+1,\ldots,j)\in\bar{G}$ have bit-length $O(\log n)$ in $\bar{\varphi}(X\cup\{y\})$.
\end{lemma}

\begin{proof}
  Call a permutation in $\bar{G}$ \emph{expressible} if it has bit-length $O(\log n)$ in $\bar{\varphi}(X\cup\{y\})$. By Lemma~\ref{rem:p+2-cycle-bit-length}, permutations in $\bar{G}$ of bounded support contained in $\{1,\ldots,p+2\}$ and cycles $(i,i+1,\ldots,j)\in\bar{G}$ with $1\leq i\leq j\leq p+2$ are expressible. Conjugating by $\bar{\varphi}(y)$ shows that all permutations of bounded support are expressible.

   Recall the explicit choice of $\bar{\varphi}(y)=\bar{y}$ made in \eqref{eq:y-bar}. If $(i,i+1,\ldots,j)\in\bar{G}$ and $k+1\leq i\leq j\leq p$, then $(i,i+1,\ldots,j)^{\bar{y}}=(k+1-j,\ldots,k-i,k+1-i)$. Hence, cycles $(i,i+1,\ldots,j)\in\bar{G}$ with $-p+k+1\leq i\leq j\leq0$ are expressible.

   Let $(i,i+1,\ldots,j)\in\bar{G}$ with $-p+k+1\leq i\leq 0$ and $1\leq j\leq p+2$. We express $(i,i+1,\ldots,j)$ as the product of three expressible permutations. If $\bar{G}=S_n$ then $(i,i+1,\ldots,j)=(1,2,\ldots,j)(0,1)(i,i+1,\ldots,0)$. Suppose that $\bar{G}=A_n$. Note that $i,j$ have the same parity. If $j$ is even, then $(i,i+1,\ldots,j)=(2,3,\ldots,j)(0,1,2)(i,i+1,\ldots,0)$. If $j$ is odd, then $(i,i+1,\ldots,j)=(1,2,\ldots,j)(-1,0,1)(i,i+1,\ldots,-1)$. Hence in all cases $(i,i+1,\ldots,j)$ is expressible.

  Let $(i,i+1,\ldots,j)\in\bar{G}$ with $i\in\{-p+k-1,-p+k\}$. If $j-i<2$ then $(i,i+1,\ldots,j)$ has bounded support, and hence is expressible. Otherwise, $(i,i+1,\ldots,j)=(i+2,\ldots,j-1,j)(i,i+1,i+2)$ is the product of two expressible permutations, and hence is also expressible.
\end{proof}

\begin{remark}
  Lemma~\ref{lem:n-cycle-bit-length} addresses all but finitely many alternating and symmetric groups. Since the statement concerns asymptotic bit-length, by Remark~\ref{rem:Theorem 3.40} it generalises to all alternating and symmetric groups of degree at least 5. Namely, for all $n\geq5$, if $\bar{G}$ is $A_n$ or $S_n$, then $\bar{J}\coloneqq\langle X\mid R \rangle\cong\bar{G}$ for some $|X|=3$, $|R|=7$, where $R$ has bit-length $O(\log n)$ in $X$. Furthermore, there exists an isomorphism $\bar{\varphi}\colon\bar{J}\to\bar{G}$ such that permutations of bounded support in $\bar{G}$ and cycles $(i,i+1,\ldots,j)\in\bar{G}$ have bit-length $O(\log n)$ in $\bar{\varphi}(X)$. This is the content of \cite[Remark 3.37]{GKKL1}.

  The implicit justification of \cite[Remark 3.37]{GKKL1} is flawed due to the errors in \cite[Theorem 3.40]{GKKL1}. The proof of the latter does not address some alternating groups (cf. Remark~\ref{rem:Theorem 3.40-wrong}). Presentations which handle the missing groups are discussed in Example~\ref{Examples3.18(1)}. These are based on \cite[Examples 3.18 (1)]{GKKL1}, which also has some errors (cf. Remark~\ref{rem:p+3-errors}).

  Although the presentations of Example~\ref{Examples3.18(1)} can be used to repair the proof of \cite[Theorem 3.40]{GKKL1}, further work is required to fully justify \cite[Remark 3.37]{GKKL1}. To this end, one could show that a claim analogous to Lemma~\ref{lem:n-cycle-bit-length} holds for the presentations of alternating groups in Example~\ref{Examples3.18(1)}. Our approach is to instead establish, using Lemma~\ref{lem:Mor}, that only finitely many groups were not handled by the proof of \cite[Theorem 3.40]{GKKL1}.
\end{remark}

\subsection{Explicit short presentations of $A_n$ and $S_n$}

In \cite[Section 3.5]{GKKL1} an explicit 3-generator 7-relator presentation of $S_n$ with bit-length $O(\log n)$ is described for $n\geq50$. However, the generators do \emph{not} satisfy the relations.

We now list a corrected version. Writing down such a presentation amounts to making specific choices of the words $\tilde{a},c,d,e,w$ as described in Theorem~\ref{thm:alt-sym-n}, each of bit-length $O(\log n)$. Some of the notation used here conflicts with that of \cite{GKKL1}.

\begin{theorem}
  \label{thm:sym-explicit}
  Let $n$ be an integer such that $14\leq n\leq 20$, $26\leq n\leq 44$ or $n\geq50$. Let $p\equiv11\bmod{12}$ be a prime such that $(n+2)/2\leq p\leq n-3$ (cf. Lemma~\ref{lem:Mor}). Let $s(r-1)\equiv-1\bmod{p}$ where $\F_p^*=\langle r \rangle$. By Theorem~\ref{thm:alt-sym-n}, the symmetric group $S_n$ has a 3-generator 7-relator presentation with bit-length $O(\log n)$. Namely:
  \[ \bar{J}\coloneqq\langle a,g,y \mid a^p = b^{p-1},\ (a^s)^b = a^{s-1},\ (zz^a)^2=1,\ \tilde{a}=(\tilde{a}h)^y,\ c=c^y,\ [d,e^y] = 1,\ y = w \rangle \cong S_n, \]
  for words $b,z,h,\tilde{a},c,d,e,w$ defined below. Let $S_n=\Sym(\{-p+k-1,\ldots,p+2\})$. Define $\bar{\varphi}(a)\coloneqq\varphi(a)$ and $\bar{\varphi}(g)\coloneqq\varphi(g)$, where $\varphi$ is defined in Remark~\ref{rem:p+2-isomorphisms}. Define $\bar{\varphi}(y)\coloneqq\bar{y}$ where $\bar{y}$ is defined in \eqref{eq:y-bar}. Then $\bar{\varphi}$ extends to an isomorphism $\bar{\varphi}\colon\bar{J}\to S_n$. Note that $a$ and $b$ play their roles as described in Theorem~\ref{thm:p+2}, and \emph{not} as in Theorem~\ref{thm:glue}. Let $k\coloneqq 2p+4-n$, so $6\leq k\leq p+1$ and $n\equiv k\bmod{2}$. As in \cite[Section 3.5]{GKKL1}, if $f$ is a word in $\{a,g,y\}$ and $\sigma$ is a permutation in $S_n$, write $f\equiv\sigma$ or $\sigma\equiv f$ to signify that $\bar{\varphi}(f)=\sigma$. Let $\alpha\in\F_p^*$ be such that $\alpha^3=r$. Such an $\alpha$ exists because $(p-1,3)=1$. Use the representatives $\{1,2,\ldots,p\}$ for $\F_p$. Define words as follows.

  \begin{enumerate}[label=(\arabic*)]
  \item \label{sym-n-gens}
    \begin{description}[noitemsep]
    \item An isomorphism $\bar{J}\cong S_n$ is induced by mapping $\{a,g,y\}$ to the following permutations (cf. Remark~\ref{rem:p+2-isomorphisms} and \eqref{eq:y-bar}):
    \item $a\equiv(1,2,\ldots,p)$,
    \item $g\equiv(x\mapsto\alpha x : x\in\F_p^*)(p,p+1,p+2)$,
    \item $y\equiv(-p-1+k,p+2)(-p+k,p+1)\cdots(-1,k+2)(0,k+1)$.
    \end{description}
  \item
    \begin{description}[noitemsep]
    \item (cf. Theorem~\ref{thm:p+2})
    \item $b\coloneqq g^3\equiv(x\mapsto rx : x\in\F_p^*)$,
    \item $z\coloneqq g^{p-1}\equiv(p,p+1,p+2)$,
    \item $h\coloneqq (b^2 z(1)z(r))^{(p+1)/2}\equiv1$.
    \end{description}
  \item
    \begin{description}[noitemsep]
    \item $z(i)\coloneqq z^{a^i}\equiv(i,p+1,p+2)$ for $1\leq i\leq p$ (cf. Lemma~\ref{rem:p+2-cycle-bit-length}),
    \item $d(i,j)\coloneqq z(i)z(j)^{-1}z(i)\equiv(i,j)(p+1,p+2)$ for $1\leq i,j\leq p$, $i\neq j$.
    \item (These remain unchanged when $i,j$ are replaced by equivalent residues mod $p$.)
    \end{description}
  \item
    \begin{description}[noitemsep]
    \item Defining even cycles (cf. Lemma~\ref{rem:p+2-cycle-bit-length}):
      % \item $x\coloneqq zz^{-a}z\equiv(1,p)(p+1,p+2)$
    \item $x\coloneqq d(0,1)\equiv(1,p)(p+1,p+2)$,
    \item $c(i,j) \coloneqq
      \begin{cases}
        a^{-j}(ax)^{j-i}a^i & \text{for } 1\leq i \leq j\leq p-1, \\
        z^{(za)^{-2}}c(i,p-2) & \text{for } 1\leq i\leq p-2 \text{ and } j=p, \\
        z^{(za)^{-1}}c(i,p-1) & \text{for } 1\leq i\leq p-1 \text{ and } j=p+1, \\
        zc(i,p) & \text{for } 1\leq i\leq p-2 \text{ and } j=p+2,
      \end{cases}$
    \item $\phantom{c(i,j){}}\equiv (i,i+1,\ldots,j)(p+1,p+2)^{j-i}$ where defined (cf. \eqref{eq:p+2-even-cycles}).
    \end{description}
  \item
    \begin{description}[noitemsep]
    \item Defining a transposition (cf. \eqref{eq:p+2-transposition-explicit}):
    \item $b_2\coloneqq b^{(p-1)/2}\equiv (1,p-1)(2,p-2)\cdots(\frac{p-1}{2},\frac{p+1}{2})$,
    \item $c_\bullet\coloneqq c(1,\frac{p-1}{2})c(\frac{p+1}{2},p-1)^{-1}\equiv(1,2,\ldots,\frac{p-1}{2})(p-1,p-2,\ldots,\frac{p+1}{2})$,
    \item $v\coloneqq (c_\bullet d(1,-1))^{(p-1)/2}\equiv(1,p-1)(2,p-2)\cdots(\frac{p-1}{2},\frac{p+1}{2})(p+1,p+2)$,
    \item $t\coloneqq vb_2\equiv(p+1,p+2)$.
    \end{description}
  \item
    \begin{description}[noitemsep]
    \item Gluing presentations (cf. Theorems~\ref{thm:glue} and \ref{thm:alt-sym-n}):
    \item $\tilde{a} \coloneqq z(3)^{z(2)z(1)} \equiv (1,2,3)$,
    \item $c \coloneqq
      \begin{cases}
        c(2,k)t\equiv (2,\ldots,k) & \text{if $n$ is odd}, \\
        c(1,k)t\equiv (1,\ldots,k) & \text{if $n$ is even},
      \end{cases}$
    \item $d\coloneqq c(5,p+2)\equiv(5,\ldots,p+2)$,
    \item $e\coloneqq
      \begin{cases}
        zat\,c(3,k+1)^{-1}\equiv(1,2,k+1,\ldots,p,p+1,p+2) & \text{if $n$ is odd}, \\
        zat\,c(2,k+1)^{-1}\equiv(1,k+1,\ldots,p,p+1,p+2) & \text{if $n$ is even}.
      \end{cases}$
    \end{description}
  \item
    \begin{description}[noitemsep]
    \item Expressing $w$ as a word in $\{a,g\}\cup \{a,g\}^y$ (cf. Theorem~\ref{thm:alt-sym-n}):
    \item If $k=p+1$, then let $w\coloneqq t^{yz}\equiv(0,p+2)\equiv y$.
    \item If $k=p$, then let $w\coloneqq z^{yz^{-1}}z^{yz}\equiv(-1,p+2)(0,p+1)\equiv y$.
    \item If $k=p-1$, then let $w\coloneqq z^{aya^{-1}z^{-1}}z^{aya^{-1}z}[d(1,p)t]^{ya^{-1}}\equiv(-2,p+2)(-1,p+1)(0,p)\equiv y$.
    \item If $k\leq p-2$, then define
    \item $\tilde{x}\coloneqq[d(1,k+2)d(2,k+1)]^{d(1,k+2)^yd(2,k+1)^y}\equiv(-1,k+2)(0,k+1)$,
    \item $u\coloneqq (za)(za)^y\equiv(1,\ldots,k,k+1,k+2,\ldots,p,p+1,p+2)(1,\ldots,k,0,-1,\ldots,-p-1+k)$,
    \item $w\coloneqq
      \begin{cases}
        (\tilde{x}u^{-2})^{(p-k)/2}\tilde{x}u^{p-k} & \text{if $n$ is odd} \\
        t^{zay(za)^{-1}}(\tilde{x}u^{-2})^{(p-k-1)/2}\tilde{x}u^{p-k-1} & \text{if $n$ is even}
    \end{cases}
    % \equiv (-p-1+k,p+2)\cdots(0,k+1)
    \equiv y
    \quad \text{(cf. \eqref{eq:horner})}$.\hfill $\begin{array}{c} \\ \qedhere \end{array}$
    \end{description}
  \end{enumerate}
\end{theorem}

\begin{proof}
  By construction the presentation has bit-length $O(\log n)$. Let $\bar{\varphi}\colon\{a,g,y\}\to S_n=\Sym(\{-p+k-1,\ldots,p+2\})$ be the mapping indicated by \ref{sym-n-gens}. Let $\varphi\colon J\to G$ and $\hat{\varphi}\colon\hat{J}\to\hat{G}$ be the isomorphisms defined in Remark~\ref{rem:p+2-isomorphisms}. The following observations, together with Theorems~\ref{thm:p+2}, \ref{thm:glue} and \ref{thm:alt-sym-n}, prove that $\bar{\varphi}$ extends to an isomorphism $\bar{\varphi}\colon\bar{J}\to S_n$.
  \begin{itemize}
  \item $\bar{\varphi}$ agrees with $\varphi$ on $X\coloneqq\{a,g\}$.
  \item $\bar{\varphi}(y)$ is as specified in \eqref{eq:y-bar}.
  \item The words $b,z,h$, and $r,s\in\F_p^*$, are as specified in Theorem~\ref{thm:p+2}.
  \item As required by Theorem~\ref{thm:glue}:
    \begin{itemize}
  \item $\hat{\varphi}(\tilde{a})=((1,2,3),1)$,
  \item $\varphi(c)\in\Sym(\{1,\ldots,k\})$,
  \item $\varphi(d)\in\Sym(\{4,\ldots,p+2\})$,
  \item $\varphi(e)^{\bar{\varphi}(y)}\in\Sym(\{-p+k-1,\ldots,3\})$,
  \item $\langle \varphi(\tilde{a}), \varphi(c) \rangle=\Sym(\{1,\ldots,k\})$,
  \item $\langle \sigma, \varphi(d) \rangle=\Sym(\{4,\ldots,p+2\})$ for some $\sigma\in\Sym(\{4,\ldots,k\})$,
  \item $\langle \varphi(a), \varphi(e)^{\bar{\varphi}(y)} \rangle=\Sym(\{-p+k-1,\ldots,3\})$,
  \item $w$ is a word in $X\cup X^y$ and $\theta(w)=\bar{\varphi}(y)$, where $\theta\colon F_{X\cup\{y\}}\to S_n$ extends $\bar{\varphi}$.\qedhere
    \end{itemize}
  \end{itemize}
\end{proof}

We also provide an explicit presentation of the alternating groups of sufficiently large degree. Here our task is somewhat simpler (we need not construct a word corresponding to a transposition). However, we must be more careful in our choice of permutation corresponding to $y$, which must be an even permutation (cf. \eqref{eq:y-bar}).

\begin{theorem}
  \label{thm:alt-explicit}
  Let $n$ be an integer such that $15\leq n\leq 20$, $27\leq n\leq 44$, or $n\geq51$. Let $p\equiv11\bmod{12}$ be such that $(n+2)/2\leq p\leq n-4$ (cf. Lemma~\ref{lem:Mor}). Let $s(r-1)\equiv-1\bmod{p}$ where $\F_p^{*2}=\langle r \rangle$. By Theorem~\ref{thm:alt-sym-n}, the alternating group $A_n$ has a 3-generator 7-relator presentation with bit-length $O(\log n)$. Namely:
  \[ \bar{J} \coloneqq \langle a,g,y \mid a^p = b^{(p-1)/2},\ (a^s)^b = a^{s-1},\ (zz^a)^2 = 1,\ \tilde{a} = (\tilde{a}h)^y,\ c = c^y,\ [d,e^y] = 1,\ y = w \rangle \cong A_n, \]
  for words $b,z,h,\tilde{a},c,d,e,w$ defined below. Let $A_n=\Alt(\{-p+k-1,\ldots,p+2\})$. Define $\bar{\varphi}(a)\coloneqq\varphi(a)$ and $\bar{\varphi}(g)=\varphi(g)$, where $\varphi$ is defined in Remark~\ref{rem:p+2-isomorphisms}. Define $\bar{\varphi}(y)\coloneqq\bar{y}$ where $\bar{y}$ is defined in \eqref{eq:y-bar}. Then $\bar{\varphi}$ extends to an isomorphism $\bar{\varphi}\colon\bar{J}\to A_n$. Note that $a$ and $b$ play their roles as described in Theorem~\ref{thm:p+2}, and \emph{not} as in the previous section. Let $k\coloneqq 2p+4-n$, so $6\leq k\leq p$ and $n\equiv k\bmod{2}$. As in \cite[Section 3.5]{GKKL1}, if $f$ is a word in $\{a,g,y\}$ and $\sigma$ is a permutation in $A_n$, write $f\equiv\sigma$ or $\sigma\equiv f$ to signify that $\bar{\varphi}(f)=\sigma$. Let $\alpha\in\F_p^{*2}$ be such that $\alpha^3=r$. Such an $\alpha$ exists because $(p-1,3)=1$. Use the representatives $\{1,2,\ldots,p\}$ for $\F_p$. Define words as follows.

  \begin{enumerate}[label=(\arabic*)]
  \item \label{alt-n-gens}
    \begin{description}[noitemsep]
    \item An isomorphism $\bar{J}\cong A_n$ is induced by mapping $\{a,g,y\}$ to the following permutations (cf. Remark~\ref{rem:p+2-isomorphisms} and \eqref{eq:y-bar}):
    \item $a\equiv(1,2,\ldots,p)$,
    \item $g\equiv(x\mapsto\alpha x : x\in\F_p^*)(p+2,p+1,p)$,
    \item $y\equiv
      \begin{cases}
        (-p-1+k,p+2)(-p+k,p+1)\cdots(-1,k+2)(0,k+1) & \text{if $n$ is odd,} \\
        (-p-1+k,p+2,-p+k,p+1)(-p+1+k,p)\cdots(-1,k+2)(0,k+1) & \text{if $n$ is even.}
      \end{cases}
      $
    \end{description}
  \item
    \begin{description}[noitemsep]
    \item (cf. Theorem~\ref{thm:p+2})
    \item $b\coloneqq g^3\equiv(x\mapsto rx : x\in\F_p^*)$,
    \item $z\coloneqq g^{(p-1)/2}\equiv(p,p+1,p+2)$,
    \item $h\coloneqq (b^2 z(1)z(-1))^{(p+1)/2}\equiv1$.
    \end{description}
  \item
    \begin{description}[noitemsep]
    \item $z(i)\coloneqq z^{a^i}\equiv(i,p+1,p+2)$ for $1\leq i\leq p$ (cf. Lemma~\ref{rem:p+2-cycle-bit-length}),
    \item $d(i,j)\coloneqq z(i)z(j)^{-1}z(i)\equiv(i,j)(p+1,p+2)$ for $1\leq i,j\leq p$, $i\neq j$.
    \item (These remain unchanged when $i,j$ are replaced by equivalent residues mod $p$.)
    \end{description}
  \item
    \begin{description}[noitemsep]
    \item Defining even cycles (cf. Lemma~\ref{rem:p+2-cycle-bit-length}):
      % \item $x\coloneqq zz^{-a}z\equiv(1,p)(p+1,p+2)$
    \item $x\coloneqq d(0,1)\equiv(1,p)(p+1,p+2)$,
    \item $c(i,j) \coloneqq
      \begin{cases}
        a^{-j}(ax)^{j-i}a^i & \text{for } 1\leq i \leq j\leq p-1, \\
        z^{(za)^{-2}}c(i,p-2) & \text{for } 1\leq i\leq p-2 \text{ and } j=p, \\
        z^{(za)^{-1}}c(i,p-1) & \text{for } 1\leq i\leq p-1 \text{ and } j=p+1, \\
        zc(i,p) & \text{for } 1\leq i\leq p-2 \text{ and } j=p+2,
      \end{cases}$
    \item $\phantom{c(i,j){}}\equiv (i,i+1,\ldots,j)(p+1,p+2)^{j-i}$ where defined (cf. \eqref{eq:p+2-even-cycles}).
    \end{description}
  \item
    \begin{description}[noitemsep]
    \item Gluing presentations (cf. Theorems~\ref{thm:glue} and \ref{thm:alt-sym-n}):
    \item $\tilde{a} \coloneqq z(3)^{z(2)z(1)} \equiv (1,2,3)$,
    \item $c \coloneqq
      \begin{cases}
        c(1,k)\equiv (1,\ldots,k) & \text{if $n$ is odd}, \\
        c(2,k)\equiv (2,\ldots,k) & \text{if $n$ is even},
      \end{cases}$
    \item $d\coloneqq c(5,p+2)\equiv(5,\ldots,p+2)$,
    \item $e\coloneqq
      \begin{cases}
        za\, c(2,k+1)^{-1}\equiv(1,k+1,\ldots,p,p+1,p+2) & \text{if $n$ is odd}, \\
        za\,c(3,k+1)^{-1}\equiv(1,2,k+1,\ldots,p,p+1,p+2) & \text{if $n$ is even}.
      \end{cases}$
    \end{description}
  \item
    \begin{description}[noitemsep]
    \item Expressing $w$ as a word in $\{a,g\}\cup \{a,g\}^y$ (cf. Theorem~\ref{thm:alt-sym-n}):
    \item If $k=p$, then let $w\coloneqq z^{yz^{-1}}z^{yz}\equiv(-1,p+2)(0,p+1)\equiv y$.
    \item If $k\leq p-1$, then define
    \item $\tilde{x}\coloneqq[d(1,k+2)d(2,k+1)]^{d(1,k+2)^yd(2,k+1)^y}\equiv(-1,k+2)(0,k+1)$,
    \item $u\coloneqq (za)(za)^y\equiv(1,\ldots,k,k+1,k+2,\ldots,p,p+1,p+2)(1,\ldots,k,0,-1,\ldots,-p-1+k)$,
    \item $w\coloneqq
      \begin{cases}
        (\tilde{x}u^{-2})^{(p-k)/2}\tilde{x}u^{p-k} & \text{if $n$ is odd}, \\
        d(1,-1)^{zaya^{-1}z^{-2}a^{-1}}z(1)^{-ya^{-1}z^{-1}}(\tilde{x}u^{-2})^{(p-k-3)/2}\tilde{x}u^{p-k-3} & \text{if $n$ is even},
      \end{cases}$
    \item $\phantom{w{}}\equiv y$ (cf. \eqref{eq:horner}).\qedhere
    \end{description}
  \end{enumerate}
\end{theorem}

\begin{proof}
  This follows from Theorems~\ref{thm:p+2}, \ref{thm:glue} and \ref{thm:alt-sym-n}. The proof is almost identical to that of Theorem~\ref{thm:sym-explicit}.
\end{proof}

\begin{remark}
  The presentations of $A_n$ and $S_n$ discussed in this section have bit-length $O(\log n)$. A stronger statement also holds. Each such presentation uses a bounded number of exponents in the relations, each of which are at most $n$.
\end{remark}

\begin{remark}
  \label{rem:freedom}
  There is some freedom in constructing the presentations of $A_n$ and $S_n$ described in this section. Let $\theta\colon F_{\{a,g,y\}}\to S_n$ map $\{a,g,y\}$ to the generators of $A_n$ or $S_n$ indicated in Theorem~\ref{thm:alt-explicit} \ref{alt-n-gens} or Theorem~\ref{thm:sym-explicit} \ref{sym-n-gens} respectively. The words $\tilde{a},c,d,e$ in $\{a,g\}$ and $w$ in $\{a,g\}\cup\{a,g\}^y$ can be replaced by words that have the same image under $\theta$. For example, $\theta(a)$ has order $p$, so if $i\equiv j\bmod{p}$, then instances of $a^i$ may be replaced by $a^j$ in the construction of these words. Such observations can be used to mildly reduce the word-length of the presentation.
\end{remark}

\begin{example}
  We construct 3-generator 7-relator presentations of $A_{17}$ and $S_{17}$ as described in Theorem~\ref{thm:alt-sym-n}. These slightly differ from the explicit presentations given in Theorems~\ref{thm:sym-explicit} and \ref{thm:alt-explicit}; they have been modified according to Remark~\ref{rem:freedom}.

  Let $n\coloneqq17$, and $p\coloneqq11$. Note that $p\equiv11\bmod{12}$ and $(n+2)/2\leq p\leq n-4$. Define $k\coloneqq 2p+4-n=9$. Let $A_{17}$ and $S_{17}$ act on $\{-3,-2,-1,0,1,\ldots,13\}$. We will use the 2-generator 4-relator presentations of $A_{11}$ and $S_{11}$ given in Theorem~\ref{thm:p+2}.

  Note that $\F_p^{*2}=\langle 5 \rangle$, and $(-3)(5-1)\equiv-1\bmod{11}$. Also $3^3\equiv5\bmod{11}$. Use $r=5$, $s=-3$ and $\alpha=3$ in Theorem~\ref{thm:p+2} to obtain
  \[ \langle a,g \mid b\coloneqq g^3,\ z\coloneqq g^5,\ h\coloneqq (b^2z^az^{a^{-1}})^6,\ a^{11}=b^5,\ (a^{-3})^b = a^{-4},\ (zz^a)^2 = 1,\ h = 1 \rangle \cong A_{11}. \]
  An isomorphism is defined by $a\mapsto(1,2,\ldots,11)$, and $g\mapsto(1,3,9,5,4)(2,6,7,10,8)(13,12,11)$ (cf. Remark~\ref{rem:p+2-isomorphisms}). Apply Theorem~\ref{thm:alt-explicit} and Remark~\ref{rem:freedom} to obtain
  \begin{align*}
    \langle a,g,y \mid b&\coloneqq g^3,\ z\coloneqq g^5,\ h\coloneqq (b^2z^az^{a^{-1}})^6,\ x\coloneqq z(z^a)^{-1}z,\ \tilde{a}\coloneqq (z^{a^3})^{z^{a^2}z^a},\ c\coloneqq a^2(ax)^{-2}a, \\
                                     & d\coloneqq zz^{(za)^{-2}}a^2(ax)^4a^5,\ e\coloneqq za^{-1}(ax)^2a^{-1},\ \tilde{y}\coloneqq z^az^{-1}z^az^{a^2}(z^{-1})^{a^{-1}}z^{a^2}, \\
                                     & \tilde{z}\coloneqq\tilde{y}^y,\ \tilde{x}\coloneqq \tilde{y}^{\tilde{z}},\ u \coloneqq (za)(za)^y,\ w\coloneqq\tilde{x}u^{-2}\tilde{x}u^2, \\
    &a^{11} = b^5,\ (a^{-3})^b = a^{-4},\ (zz^a)^2=1,\ \tilde{a}=(\tilde{a}h)^y,\ c=c^y,\ [d,e^y]=1,\ y=w \rangle \cong A_{17}.
  \end{align*}
  An isomorphism is defined by $a\mapsto(1,2,\ldots,11)$, $g\mapsto(1,3,9,5,4)(2,6,7,10,8)(13,12,11)$, and $y\mapsto(-3,13)(-2,12)(-1,11)(0,10)$.

  Note that $\F_p^*=\langle 2 \rangle$, and $(-1)(2-1)\equiv-1\bmod{11}$. Also $7^3\equiv2\bmod{11}$. Use $r=2$, $s=-1$ and $\alpha=7$ in Theorem~\ref{thm:p+2} to obtain
  \[ \langle a,g,y \mid b\coloneqq g^3,\ z\coloneqq g^{10},\ h\coloneqq (b^2z^az^{a^2})^6,\ a^{11} = b^{10},\ (a^{-1})^b = a^{-2},\ (zz^a)^2=1,\ h=1\rangle \cong S_{11}. \]
  An isomorphism is defined by $a\mapsto(1,2,\ldots,11)$, and $g\mapsto(1,7,5,2,3,10,4,6,9,8)(11,12,13)$ (cf. Remark~\ref{rem:p+2-isomorphisms}). Apply Theorem~\ref{thm:sym-explicit} and Remark~\ref{rem:freedom} to obtain
  \begin{align*}
    \langle a,g,y \mid b&\coloneqq g^3,\ z\coloneqq g^{10},\ h\coloneqq (b^2z^az^{a^2})^6,\ x\coloneqq z(z^a)^{-1}z,\ \tilde{a}\coloneqq (z^{a^3})^{z^{a^2}z^a}, \\
    & c_\bullet\coloneqq a^{-5}(ax)^4a^{-5}(ax)^{-4}a^{-1},\ v \coloneqq (c_\bullet z^a(z^{-1})^{a^{-1}}z^a)^5,\ t\coloneqq vb^5,\ c\coloneqq a^2(ax)^{-3}a^2t, \\
                                     & d\coloneqq zz^{(za)^{-2}}a^2(ax)^4a^5,\ e\coloneqq  zata^{-3}(ax)^3a^{-1},\ \tilde{y}\coloneqq z^az^{-1}z^az^{a^2}(z^{-1})^{a^{-1}}z^{a^2}, \\
                                     & \tilde{z}\coloneqq\tilde{y}^y,\ \tilde{x}\coloneqq \tilde{y}^{\tilde{z}},\ u \coloneqq (za)(za)^y,\ w\coloneqq\tilde{x}u^{-2}\tilde{x}u^2, \\
    &a^{11} = b^{10},\ (a^{-1})^b = a^{-2},\ (zz^a)^2=1,\ \tilde{a}=(\tilde{a}h)^y,\ c=c^y,\ [d,e^y]=1,\ y=w \rangle \cong S_{17}.
  \end{align*}
  An isomorphism is defined by $a\mapsto(1,2,\ldots,11)$, $g\mapsto(1,7,5,2,3,10,4,6,9,8)(11,12,13)$, and $y\mapsto(-3,13)(-2,12)(-1,11)(0,10)$.
\end{example}

\subsection{Providing access to the presentations}

The presentations and generators of alternating and symmetric groups in Example~\ref{Examples3.18(1)} (cf. \eqref{eq:alt-p+3}), Theorem~\ref{thm:p+2} (cf. \eqref{eq:p+2}), and Theorems~\ref{thm:sym-explicit} and \ref{thm:alt-explicit} are publicly available \cite{github} in \textsc{Magma}. Relators are stored as straight-line programs \cite{BS}.

Evaluating the relations readily demonstrates that, for any input degree supplied by the user, the alternating or symmetric group is a quotient of the corresponding finitely presented group.

Coset enumeration does not work well for the presentations in Theorems~\ref{thm:sym-explicit} and \ref{thm:alt-explicit}. For instance, in minimum degree examples, no coset enumeration over a subgroup corresponding to a point stabilizer completed.

\printbibliography

\end{document}